%% file: NC+AS.tex
\documentclass[11pt]{article}
\usepackage{paper,asb}

\usepackage{cite}
\usepackage{soul}

\newcommand{\papertitle}{Exploiting Negative Curvature in Conjunction with Adaptive Sampling: Theoretical Results and a Practical Algorithm}
\newcommand{\paperauthora}{Albert S. Berahas}
\newcommand{\paperauthoraaffiliation}{Department of Industrial and Operations Engineering, University of Michigan}
\newcommand{\paperauthoraemail}{albertberahas@gmail.com}
\newcommand{\paperauthorb}{Wanping Dong}
\newcommand{\paperauthorbemail}{wanpingd@umich.edu}
\newcommand{\paperauthorc}{Raghu Bollapragada}
\newcommand{\paperauthorcaffiliation}{Operations Research and Industrial Engineering Program, University of Texas at Austin}
\newcommand{\paperauthorcemail}{raghu.bollapragada@utexas.edu}

\begin{document}
\title{\papertitle}

\author{\paperauthora \footnotemark[2] \footnotemark[1]
   \and \paperauthorc\footnotemark[3]\and \paperauthorb\footnotemark[2]\ }

\maketitle

\renewcommand{\thefootnote}{\fnsymbol{footnote}}
\footnotetext[2]{\paperauthoraaffiliation. (\url{\paperauthoraemail},\url{\paperauthorbemail})}
\footnotetext[3]{\paperauthorcaffiliation. (\url{\paperauthorcemail})}
\footnotetext[1]{Corresponding author.}
\renewcommand{\thefootnote}{\arabic{footnote}}

\begin{abstract}{
\input{abstract.tex}}
\end{abstract}

\input{body.tex}

\newpage
\bibliographystyle{plain}
\bibliography{references}


\end{document}

%% file: abstract.tex
In this paper, we propose algorithms that exploit negative curvature for solving noisy nonlinear nonconvex unconstrained optimization problems. We consider both deterministic and stochastic inexact settings, and develop two-step algorithms that combine directions of negative curvature and descent directions to update the iterates. Under reasonable assumptions, we prove second-order convergence results and derive complexity guarantees for both settings. To tackle large-scale problems, we develop a practical variant that utilizes the conjugate gradient method with negative curvature detection and early stopping to compute a step, a simple adaptive step size scheme, and a strategy for selecting the sample sizes of the gradient and Hessian approximations as the optimization progresses. Numerical results on two machine learning problems showcase the efficacy and efficiency of the practical method.


%% file: body.tex
\section{Introduction} \label{sec.intro}

In this paper, we consider the unconstrained optimization problem
\begin{equation}\label{prob.minf}
    \min_{x \in \R{n}} f(x)
\end{equation}
where $f:\R{n} \to \R{}$ is twice continuously differentiable and possibly nonconvex. We consider settings in which we only have deterministic or stochastic inexact estimates of the objective function and its associated derivatives. 
In the deterministic inexact setting, estimates of the objective function and its associated derivatives are contaminated with deterministic and bounded noise. 
In the stochastic setting, $f(x) = \mathbb{E}_{\xi} [F(x,\xi)]$, where $\xi$ is a random variable with associated probability space $(\Xi, \mathcal{F}, \Prob)$ and $F: \R{n} \times \Xi \to \R{}$. Deterministic and stochastic problems of this form arise in many applications, e.g., energy systems \cite{dincer2017optimization, wood2013power}, robotics \cite{lavalle2006planning, lengagne2013generation}, engineering design \cite{bendsoe2013topology, rao2019engineering}, signal processing \cite{haykin2002adaptive, proakis2007digital}, chemical engineering \cite{biegler2010nonlinear, smith2005chemical}, and machine learning \cite{goodfellow2016deep, jolliffe2002principal}. Nonconvex optimization problems are generally NP-hard~\cite{nocedal1999numerical, pardalos1991quadratic} as they can have multiple local minima/maxima and saddle points, making it difficult for optimization algorithms to navigate the landscape and distinguish between global and local optima. The situation only becomes more complicated in inexact settings. Thus, we focus on developing algorithms whose goal is to find second-order stationary points, i.e., points $x \in \R{n}$ that satisfy
\begin{equation*}
    \nabla f(x) = 0 \text{ and } \nabla^2 f(x) \succeq 0.
\end{equation*}
Specifically, our goal is to find ($\epsilon_g,\epsilon_H$)-second-order stationary points, i.e., points $x \in \R{n}$ that satisfy
\begin{equation} \label{goal.stationary}
    \|\nabla f(x)\|_2 \leq \epsilon_g \quad \text{and} \quad \nabla^2 f(x) \succeq -\epsilon_H I,
\end{equation}
where $\epsilon_g>0$ and $\ \epsilon_H > 0$; see e.g.,~\cite{murty1987some, jin2017escape,royer2018complexity,cao2023first}).

Numerous gradient-based first- and second-order algorithms have been proposed to solve such problems; see e.g., \cite{nocedal1999numerical,bertsekas1997nonlinear,bottou2018optimization} and the references therein. Most of the proposed methods for the inexact settings are first-order, i.e., they utilize only gradient information, and only guarantee convergence to first-order stationary points and/or struggle to escape from saddle points \cite{du2017gradient}. To guarantee higher-order (second-order) convergence, one needs to utilize some form of curvature (second-order) information \cite{nocedal1999numerical}. Examples of algorithms that exploit directions of negative curvature, i.e., directions that reflect the negative definiteness of the Hessian matrix, can be found in~\cite{curtis2019exploiting, forsgren1995computing, liu2017noisy, more1979use, reddi2018generic}. These methods converge to second-order stationary points, escape from saddle points efficiently, and are endowed with fast local convergence guarantees \cite{curtis2019exploiting, liu2017noisy}.

The information contained in the negative eigenvalues of the Hessian matrix is the centerpiece of negative curvature methods. A simple example is the modified Newton method \cite{fletcher1977modified}, developed to explicitly avoid directions of negative curvature by appropriately modifying the Hessian matrix. In 
\cite{more1979use, goldfarb1980curvilinear, forsgren1995computing} the authors proposed algorithms that employ a linear combination of descent and negative curvature directions to update the iterate and established convergence guarantees to second-order stationary points. Similar convergence guarantees are derived for methods that utilize one type of step (descent or negative curvature) successively, until the iterates reach one of the two conditions of an approximate second-order stationary point \eqref{goal.stationary}, and then switch to the other type of step \cite{carmon2018accelerated, li2023randomized,liu2018adaptive}. While such methods are simple and convergent, they often have additional hyper-parameters and are often inferior in practice. 
In \cite{curtis2019exploiting} a two-step methodology was proposed where both descent and negative curvature steps are taken at every iteration. Adaptive schemes for selecting the type of step have also been proposed \cite{curtis2019exploiting,royer2018complexity,steihaug1983conjugate}. Specifically, in \cite{curtis2019exploiting} the type of step at a given iteration is determined by comparing the potential progress of the two steps, and, in \cite{royer2018complexity, steihaug1983conjugate} the conjugate gradient (CG) method with negative curvature detection and early stopping is used as a subroutine to compute a step.

While the benefits of utilizing directions of negative curvature are many, they are less studied and less utilized as compared to gradient-based (first-order) methods. This is primarily because these methods are more complex and that directions of negative curvature do not come for free and can be expensive \cite{nocedal1999numerical, zhang2020symmetry}. The situation invariably becomes more challenging in the inexact (deterministic and stochastic) settings. In this paper, we design and analyze two-step negative curvature methods for the inexact settings, and develop a practical variant that consists of three main components: 
(1) an adaptive sampling strategy that dictates the accuracy in the approximations employed (to develop practical and convergent algorithms); (2) the matrix-free 
CG method with negative curvature detection and early stopping (to efficiently compute search directions); and (3) a simple strategy for dynamically selecting the step size (to alleviate the need for tuning hyperparameters). In the remainder of the introduction, we discuss the literature most closely related to our work and outline the three main components of our proposed practical approach.

Several works propose negative curvature method with exact gradient and Hessian information \cite{carmon2018accelerated, fletcher1977modified, forsgren1995computing, goldfarb1980curvilinear, more1979use, reddi2018generic, royer2018complexity}. In \cite{liu2017noisy, xu2020newton}, the authors develop negative curvature methods in the exact gradient and (deterministic) inexact Hessian setting. In the former, the authors propose an algorithm that selects between an inexact negative curvature direction and an exact gradient direction at every iteration and prove that this approach converges to an ($\epsilon_g$, $\epsilon_H$)-second-order stationary point, and in the latter, the authors propose cubic regularization and trust-region methods with inexact Hessians and exact gradients that enjoy the same iteration complexity as their deterministic counterparts (exact gradients and Hessians). We note that these methods require strong conditions on the Hessian approximations employed in order to obtain sufficiently accurate directions of negative curvature. Moreover, both of these methods make use of the exact gradient in order to ensure that the directions are indeed descent directions. 
In the stochastic setting, methods that leverage directions of negative curvature \cite{curtis2019exploiting, liu2017noisy, liu2018adaptive} are even less studied. These works impose conditions analogous to the classical conditions for the stochastic gradient (SG) method, 
and as a result, provide convergence guarantees analogous to those of the SG method \cite[Section 4]{bottou2018optimization} and do not provide any second-order convergence guarantees. The randomized algorithm proposed in \cite{li2023randomized} selects the directions of negative curvature using a randomized approach (coin flipping) and is endowed with expected and high probability second-order complexity guarantees. 
That said, the results rely on strong conditions on the gradient and Hessian approximations, and are agnostic to the accuracy in the gradient approximations utilized. 

To tackle large-scale problems and develop a practical algorithm that utilizes directions of negative curvature, we leverage adaptive sampling techniques~\cite{byrd2012sample}, the power of the matrix-free CG~\cite{hestenes1952methods} method with negative curvature detection and early stopping~\cite{royer2020newton}, and a simple adaptive line search procedure~\cite{bollapragada2018adaptive}. At every iteration, our method requires an estimate of the gradient and Hessian whose accuracy is dictated by an adaptive sampling mechanism. The idea is simple yet powerful; adjust the accuracy in the approximations employed as required over the course of the optimization. 
The key to such approaches is the mechanism and rule for the selection of the accuracy. In the unconstrained stochastic setting, adaptive sampling algorithms have proven successful, e.g., norm condition \cite{carter1991global,byrd2012sample,berahas2021global} and inner-product test \cite{bollapragada2018adaptive,bollapragada2018progressive}. The accuracy (samples sizes) in the approximations employed in these methods are dictated by variance estimates that ensure sufficient decrease 
in expectation. We note that in \cite[Section 5]{byrd2012sample} a practical Newton-CG 
algorithm  that uses adaptive sampling in computing  function, gradient, and Hessian information is proposed for strongly convex problems (no negative curvature). Given gradient and Hessian estimates, our method also uses the CG method as a subroutine with negative curvature detection and early stopping. The reasons for this choice are threefold: $(i)$ good for large-scale as it can be implemented matrix-free; $(ii)$ easy to incorporate gradient and Hessian approximations; and, $(iii)$ each iteration of CG can check and detect possible direction of negative curvature at no additional cost. As a result the subroutine terminates either with an approximation to Newton's direction or a direction of negative curvature. Some examples of methods that utilize the CG method as a subroutine are \cite{byrd2011use,byrd2012sample,royer2020newton,berahas2020investigation}. The final component is a dynamic step size procedure, similar to that utilized in \cite{bollapragada2018progressive, friedlander2012hybrid, beck2009fast}, which is well suited for the adaptive sampling setting.

The paper is organized as follows. We conclude this section by summarizing our contributions and introducing the notation used throughout the paper. In Section~\ref{sec.inexact}, we consider the deterministic inexact setting, provide conditions on the approximations computed and the search directions employed, and establish convergence guarantees. In Section \ref{sec.stoch_setting}, we pivot to the stochastic setting, provide analogous conditions to those in Section~\ref{sec.inexact}, and establish convergence guarantees in expectation for constant and diminishing step sizes. We describe our practical algorithm and the key components in Section~\ref{sec.practical}. In Section \ref{sec.num_exp}, we present numerical results on two nonconvex machine learning tasks and illustrate the efficiency and robustness of our approach. Concluding remarks are given in Section~\ref{sec.fin_rem}.

\subsection{Contributions}

Our main contributions are three-fold.
\begin{enumerate}
    \item In the deterministic inexact setting, we propose a two-step algorithmic framework that employs both descent and negative curvature steps. We consider deterministic inexactness conditions on the gradient and Hessian approximations computed and the search directions employed, and establish convergence and iteration complexity guarantees analogous to those in the setting in which exact quantities (gradient and Hessian evaluations) can be computed. Specifically, under constant sufficiently small step sizes choices, the iterates generated 
    converge to second-order stationary points.
    \item In the stochastic setting, we consider the same two-step algorithmic framework under stochastic conditions on the gradient and Hessian approximations and the search directions. We consider conditions in expectation on the gradient approximations, analogous to those imposed on the SG method, as well as conditions in expectation on the Hessian approximations, and provide expected convergence and complexity results for different step size and inexactness conditions. Specifically, in the constant variance and diminishing step size regime we show convergence analogous to that of the SG method, and in the diminishing variance and constant step size regime we show expected convergence to 
    second-order stationary points and provide expected iteration complexity guarantees.
    \item Motivated by the power of utilizing negative curvature information and the success of adaptive sampling strategies in optimization, we propose a practical adaptive sampling negative curvature conjugate gradient method for the large-scale setting. We propose a mechanism for selecting the sample sizes of the gradient and Hessian adaptively as the optimization progresses. The practical method utilizes the matrix-free CG method with negative curvature detection and early termination to compute a step, and a simple dynamic step size selection procedure. The robustness and efficiency of our practical algorithm and the importance of all key components is illustrated on two nonconvex machine learning tasks. 
\end{enumerate}

\subsection{Notation}

Our proposed algorithms generate a sequence of iterates $\{ x_k\}$ with $x_k \in \R{n}$ for all $k \in \N{}$. The algorithms involve directions of negative curvature denoted by $p_k \in \R{n}$ and descent directions denoted by $d_k \in \R{n}$. These search directions are computed based on noisy (deterministic or stochastic) gradient and Hessian approximations, denoted by $g_k \in \R{n}$ and $H_k \in \R{n \times n}$, respectively. The left-most eigenvalues of the matrices $\nabla^2 f(x_k)$ and $H_k$ are denoted by $\lambda_{\min}(\nabla^2 f(x_k))$ and $\lambda_k$, respectively.

\section{Deterministic Inexact Setting} \label{sec.inexact}

In this section, we consider \eqref{prob.minf} in the setting in which only inexact gradient and Hessian approximations with deterministic noise are available. We adapt the two-step algorithmic framework  \cite[Algorithm 1]{curtis2019exploiting}, and extend it to the inexact setting. As the name suggests, the two-step framework takes two types of steps at each iteration: $(i)$ a negative curvature step $p \in \mathbb{R}^n$ when available, where $p^{\mathsf{T}} \nabla^2 f(x_k) p < 0$, and $(ii)$ a descent step $d \in \mathbb{R}^n$, where $d^{\mathsf{T}} \nabla f(x_k) < 0$. 
A generic version of this framework is given in Algorithm \ref{alg.frame.two-step}.
\begin{algorithm}
    \caption{Generic Two-Step Method} \label{alg.frame.two-step}
    \begin{algorithmic}[1]
        \Require{$x_0 \in \Real^n$, $\alpha_k > 0 $, $\beta_k > 0$}
        \ForAll{$k \in \{0, 1, \dots\}$} 
            \State choose negative curvature direction $p_k$ \label{line.negative}
            \State set $\hat{x}_k \leftarrow x_k+ \beta_k p_k$
            \State choose descent direction $d_k$ \label{line.descent}
    	\If{$p_k = d_k=0$} 
        	\Return $x_k$
            \EndIf
            \State set $x_{k+1} \leftarrow \hat{x}_k + \alpha_k d_k = x_k + \beta_k p_k + \alpha_k d_k$
	\EndFor
    \end{algorithmic}
\end{algorithm}

Before we proceed to present and analyze our proposed algorithm, we provide a simple example of negative curvature and descent steps.
\bexample
At iteration $k \in \mathbb{N}$, given the iterate $x_k \in \mathbb{R}^n$, the two steps in Lines~\ref{line.negative} and \ref{line.descent} of Algorithm~\ref{alg.frame.two-step} can be chosen as follows.
\begin{itemize}
    \item \textbf{Negative curvature direction:} If $\nabla^2 f(x_k) \succeq 0$, then $p_k \leftarrow 0$; otherwise, $p_k$ can be chosen as the eigenvector corresponding to the minimum eigenvalue of $\nabla^2 f(x_k)$ such that $p_k^{\mathsf{T}} \nabla^2 f(x_k) p_k < 0$ and $\nabla f(x_k)^Tp_k<0$.
    \item \textbf{Descent direction:} The descent direction can be set as $d_k \leftarrow - W_k \nabla f(\hat{x}_k)$ for some $w_1 I \preceq W_k \preceq w_2 I$ where $0 < w_1 \leq w_2$.
\end{itemize}
Using these conditions, one can show that the iterates generated by Algorithm \ref{alg.frame.two-step}, with appropriately chosen step size sequences, converge to second-order stationary points \cite{curtis2019exploiting}. Note that in Algorithm \ref{alg.frame.two-step}, descent steps are taken at every iteration, whereas negative curvature directions are only taken when negative curvature is present.
\eexample

With the two-step algorithmic framework in mind, the remainder of this section is structured as follows. Conditions for the two directions ($p_k$ and $d_k$) are provided in Section \ref{subsec.inexact.conditions}, and the proposed algorithm is given in Algorithm \ref{alg.inexact.two-step}. The associated convergence and complexity guarantees are provided in Section \ref{subsec.inexact.conv}.

\subsection{Assumptions, Conditions, and Algorithm} \label{subsec.inexact.conditions}

Throughout the paper, we make the following assumption. 
\bassumption \label{asm.function}
The function $f: \R{n} \rightarrow \R{}$ is continuously differentiable, the gradient of $f$ is $L_g$-Lipschitz continuous, and the Hessian of $f$ is $L_H$-Lipschitz continuous, for all $x \in \R{n}$. Moreover, the function $f$ is bounded below by a scalar $\bar{f} \in \R{}$.
\eassumption
\noindent
The above assumption is standard in the literature of methods that exploit negative curvature, and more generally second-order methods; see e.g., \cite{bellavia2020high, curtis2019exploiting, nocedal1999numerical}. 

With regards to the objective function and its associated derivatives, we assume that those quantities are prohibitively expensive to compute at every iteration and that only inexact approximations are available. Under these inexact approximations, the search directions are required to satisfy the following conditions. The first two conditions pertain to the directions of negative curvature, and set certain requirements on the inexact gradient and Hessian approximations ($g_k$ and $H_k$, respectively). 

\bcondition \label{cond.inexact.pk}
For all $k\in \mathbb{N}$, if $\lambda_k \geq 0$ (left-most eigenvalue of $H_k$), the negative curvature direction is $p_k \leftarrow 0$. Otherwise, the negative curvature direction is 
\begin{equation}\label{fom.inexact.pkselect}
    p_k = \begin{cases}
        q_k, & q_k^{\mathsf{T}}g_k \leq 0 \\
        -q_k, & \text{otherwise}
    \end{cases}
\end{equation}
where ${q}_k \in \mathbb{R}^n$ is a negative curvature vector (with respect to the matrix $H_k$) such that
\begin{subequations}\label{fom.inexact.qk}
    \begin{align}
         q_k^{\mathsf{T}} H_k q_k &\leq \gamma \lambda_k \|q_k\|_2^2 < 0, \quad \gamma \in (0, 1] \label{fom.inexact.qk1} \\
         \|q_k\|_2 &= \delta |\lambda_k|, \quad\quad \quad \quad\delta \in (0, \infty).  \label{fom.inexact.qk2}
    \end{align}
\end{subequations}
and $g_k \in \mathbb{R}^n$, the gradient approximation, 
satisfies, 
\begin{equation} \label{fom.inexact.gk}
    |(g_k - \nabla f(x_k))^{\mathsf{T}} q_k| \leq \theta |g_k^{\mathsf{T}} q_k|, \quad \theta \in [0, 1).
\end{equation}
\econdition
\bremark We make a few remarks about Condition~\ref{cond.inexact.pk}. When $\lambda_k \geq 0$, the matrix $H_k$ is positive semi-definite, i.e., there are no directions of negative curvature, and thus $p_k \gets 0$. When $\lambda_k < 0$, a direction that is sufficiently negative definite \eqref{fom.inexact.qk1} with controlled norm \eqref{fom.inexact.qk2} is guaranteed to exist. The sign of the direction of negative curvature is decided via \eqref{fom.inexact.pkselect}, and combined with the gradient accuracy requirement \eqref{fom.inexact.gk} 
ensures that this direction is a descent direction for the true objective function. Inequality \eqref{fom.inexact.gk} controls the accuracy of $g_k$ along the direction $q_k$. Since the exact gradient is not available, one can also select the direction of negative curvature \eqref{fom.inexact.pkselect} via a randomized approach (e.g., coin flipping) as proposed in~\cite{li2023randomized}. However, such approaches are inefficient in adaptive sampling settings where the accuracy in the gradient approximations improves as the optimization progresses, and negative curvature directions can be more accurately computed using the gradient approximation as simply flipping a coin can be wasteful. 
\eremark

Condition \ref{cond.inexact.pk} assumes access to an inexact Hessian approximation $H_k$ to compute the  negative curvature direction $p_k$. The following condition on $H_k$ is required for the analysis.
\bcondition \label{cond.inexact.hk}
For all $k\in \mathbb{N}$, the approximate Hessian $H_k$ and the negative curvature direction $p_k$ satisfy
\begin{equation} \label{fom.inexact.hk}
    \|(H_k - \nabla^2 f(x_k) ) p_k\|_2 \leq \gamma_H |\lambda_k^{-}| \|p_k\|_2, \quad \gamma_H \in [0, \gamma),
\end{equation}
where $\lambda_k^- := \min\{\lambda_k, 0\}$. Moreover, the gap between the left-most eigenvalues of $H_k$ and $\nabla^2 f(x_k)$ is bounded by
\begin{equation} \label{fom.inexact.lbdk}
    |\lambda_k^{-} - \lambda_{\min, k}^-| \leq \gamma_{\lambda} |\lambda_k^{-}|, \quad \gamma_{\lambda} \in [0, 1),
\end{equation}
where $\lambda_{\min, k}^- := \min\{\lambda_{\min}(\nabla^2 f(x_k)), 0 \}$. 
\econdition

\bremark We make a few remarks about Condition~\ref{cond.inexact.hk}. 
\begin{itemize}
    \item The condition in \eqref{fom.inexact.hk} ensures that the Hessian approximation $H_k$ is sufficiently accurate along the negative curvature direction $p_k$, and is weaker than  similar conditions on $\|H_k - \nabla^2 f(x_k)\|_2$ \cite{cartis2018global, li2023randomized}. 
    We note that \eqref{fom.inexact.hk} is equivalent to
    \begin{equation} \label{fom.inexact.hk_q}
        \|(H_k - \nabla^2 f(x_k) ) q_k\|_2 \leq \gamma_H |\lambda_k| \|q_k\|_2.
    \end{equation}
    When $\lambda_k < 0$, by definition $\lambda_k^{-} = \lambda_k$, and \eqref{fom.inexact.hk} is equivalent to $\|(H_k - \nabla^2 f(x_k) ) p_k\|_2 \leq \gamma_H |\lambda_k| \|p_k\|_2$. Otherwise ($\lambda_k \geq 0$), 
    $p_k = 0$ and $\lambda_k^{-} = 0$, and \eqref{fom.inexact.hk} holds trivially. 
    \item The condition in \eqref{fom.inexact.lbdk} ensures that the left-most eigenvalue of the Hessian approximation has the same sign as the left-most eigenvalue of the true Hessian, when the left-most eigenvalue of the true Hessian is negative. The condition is weaker than $|\lambda_k - \lambda_{\min, k} | \leq \gamma_{\lambda} |\lambda_k|$ as it requires the eigenvalues to be sufficiently close only when they are negative. 
\end{itemize}
\eremark

We conclude the discussion related to the negative curvature directions by describing a procedure for computing the negative curvature direction $p_k$. First, a Hessian approximation $H_k$ is formed and a direction of negative curvature (with respect to this matrix) $q(H_k)$ is computed that satisfied \eqref{fom.inexact.qk} (for example, by the matrix-free Lanczos method \cite{royer2020newton}). If $H_k$ and $q(H_k)$ satisfy inequality \eqref{fom.inexact.hk_q}, set $q_k \gets q(H_k)$ and proceed to compute $g_k$ that satisfies \eqref{fom.inexact.gk} and set $p_k$ via \eqref{fom.inexact.pkselect}. If \eqref{fom.inexact.hk_q} is not satisfied, refine the Hessian approximation and compute a new direction $q(H_k)$.

Given a direction of negative curvature, the iterate is updated via $\hat{x}_k \leftarrow x_k + \beta p_k$, where $\beta>0$ is the step size, and the algorithm proceeds to the descent step. At the point $\hat{x}_k$, a descent direction is obtained by approximating the gradient. For simplicity, we assume $d_k \leftarrow - \hat{g}_k$, where $\hat{g}_k$ is a gradient approximation at $\hat{x}_k$ that satisfies the following condition. 
\bcondition \label{cond.inexact.dk}
For all $k\in \mathbb{N}$, the 
gradient approximation $\hat{g}_k$ satisfies 
\begin{equation} \label{fom.inexact.hatgk}
    \|\nabla f(\hat{x}_k)-\hat{g}_k \|_2 \leq \hat{\theta} \|\hat{g}_k\|_2, \quad \hat{\theta} \in [0, 1).
\end{equation}
\econdition
\bremark We make a few remarks about Condition~\ref{cond.inexact.dk}, a variant of the well-established norm condition \cite{carter1991global,byrd2012sample}. This condition guarantees that $d_k= - \hat{g}_k$ is a descent direction with respect to the objective function at $\hat{x}_k$, i.e., $d_k^{\mathsf{T}} \nabla f(\hat{x}_k) < 0$. Moreover, it forces the norm of $d_k$ to be close to the norm of $\nabla f(\hat{x}_k)$. 
We note that our analysis and presentation could be generalized to directions $d_k = - W_k \hat{g}_k$, where $w_1 I \preceq W_k \preceq w_2 I$ and $0 < w_1 \leq w_2$, and  $\|H_k(\nabla f(\hat{x}_k) - \hat{g}_k)\| \leq \hat{\theta} \|H_k \hat{g}_k\|_2$ instead of \eqref{fom.inexact.hatgk}. A similar condition on a general descent direction $d_k$ is imposed in \cite{curtis2019exploiting}
\begin{equation*}
    \begin{split}
        d_k^{\mathsf{T}} \nabla f(x_k) & \geq \nu_1 \|d_k\|_2 \|\nabla f(x_k)\|_2, \quad \nu_1 \in (0, \infty), \\
        \nu_2 \|\nabla f(x_k)\|_2 & \leq \|d_k\|_2 \leq \nu_3 \|\nabla f(x_k)\|_2, \quad \nu_2 \in (0, \infty),\ \nu_3 \in [\nu_2, \infty).
    \end{split}
\end{equation*}
These conditions and \eqref{fom.inexact.hatgk} are closely related, and in fact equivalent for certain parameter settings. We use \eqref{fom.inexact.hatgk} in our presentation and analysis as it forms the basis for our practical algorithm and the adaptive sampling strategy for the gradient approximations.
\eremark

Before we proceed to present, discuss and analyze the complete algorithm (
Algorithm \ref{alg.inexact.two-step}), we acknowledge that Conditions~\ref{cond.inexact.pk}, \ref{cond.inexact.hk}, and~\ref{cond.inexact.dk} are relatively strong. 
That said, we impose them for two main reasons. First, to understand permissible errors while retaining deterministic-type results. And, second, our motivation is to develop practical adaptive sampling methods for which the accuracy in the approximations employed over the course of the optimization can be adjusted, and in those settings the assumptions and conditions above are not unrealistic. For example, consider the finite-sum setting and an algorithm where the accuracy in the gradient and Hessian approximations employed is increased at every iteration.

\begin{algorithm}[h]
    \caption{Two-Step Method}\label{alg.inexact.two-step}
    \begin{algorithmic}[1]
        \Require{$x_0 \in \Real^n$, $\alpha > 0$, $\beta > 0$
        }
        \ForAll{$k \in \{0, 1, \dots\}$} 
	\State compute Hessian approximation $H_k$ that satisfies Condition \ref{cond.inexact.hk}
            \If{$\lambda_k \geq 0$}\label{line.term_negative}
                \State set $p_k \leftarrow 0$ 
            \Else 
                \State compute $q_k$ and set $p_k$ using Condition \ref{cond.inexact.pk}
            \EndIf
            \State set $\hat{x}_k \leftarrow x_k+ \beta p_k$
            \State compute gradient approximation $\hat{g}_k$ that satisfies Condition \ref{cond.inexact.dk}
            \If{$\hat{g}_k = 0$} \label{line.term_descent}
                \State set $d_k \leftarrow 0$
            \Else
                \State set $d_k \leftarrow - \hat{g}_k$
	    \EndIf
    	\If{$p_k = d_k=0$} 
        	\Return $x_k$ \label{line.terminate}
            \EndIf
            \State set $x_{k+1} \leftarrow \hat{x}_k + \alpha d_k = x_k + \beta p_k + \alpha d_k$
	\EndFor
    \end{algorithmic}
\end{algorithm}
\bremark We make a few remarks about Algorithm~\ref{alg.inexact.two-step}. 
\begin{itemize}
    \item \textbf{Step size:} The constant step size choices (for both steps) in Algorithm~\ref{alg.inexact.two-step} are required to be  sufficiently small in order to guarantee convergence. The specific ranges for ensuring second-order convergence depend on user and problem-specific parameters and the exact forms are given in Theorem \ref{thm.inexact.converge}.
    \item \textbf{Termination conditions/Tolerances:} Algorithm~\ref{alg.inexact.two-step} terminates on Line~\ref{line.terminate} with a second-order stationary point. To guarantee convergence to an $(\epsilon_g, \epsilon_H)$-approximate second-order stationary point ($\epsilon_g \geq 0$, $\epsilon_H \geq 0$), one can relax the termination conditions Lines~\ref{line.term_negative} and \ref{line.term_descent} of Algorithm~\ref{alg.inexact.two-step} to $\lambda_k \geq -\frac{\epsilon_H}{1 + \gamma_{\lambda}}$ and $\|\hat{g}_k\|_2 \leq \frac{\epsilon_g}{1+\hat{\theta}}$, respectively. 
    \item \textbf{Computations \& Practicality:} 
    At every iteration, a straightforward implementation of Algorithm~\ref{alg.inexact.two-step} requires a Hessian evaluation and the computation of the minimum eigenvalue of the Hessian matrix to compute the negative curvature step, and a gradient evaluation to compute the descent step. 
    In the large-scale setting, these computations (and in particular the two first computations) can be prohibitively expensive and/or storage intensive. There are multiple ways to compute a direction of  negative curvature. In theory, the most straightforward approach is to set $p_k$ to be the eigenvector corresponding to the minimum eigenvalue of the Hessian matrix. In practice, and to avoid expensive computations, the negative curvature direction can be computed via a matrix-free iterative approach e.g., matrix-free Lanczos method \cite{lanczos1950iteration}. In our practical method (Section~\ref{sec.practical}) we make use of the CG method with negative curvature detection. We note that the negative curvature detection comes at no additional cost.
\end{itemize} 
\eremark

\subsection{Convergence and Complexity Results} \label{subsec.inexact.conv}

In this subsection, we present convergence and complexity guarantees for Algorithm \ref{alg.inexact.two-step} using constant step sizes in the deterministic inexact setting. The main theorem (presented below) shows that the iterates generated by Algorithm \ref{alg.inexact.two-step} with sufficiently small step sizes converge to a second-order stationary point. 

\btheorem\label{thm.inexact.converge}
    Suppose Assumption \ref{asm.function} and Conditions \ref{cond.inexact.pk}, \ref{cond.inexact.hk} and \ref{cond.inexact.dk} hold. Let the step size parameters satisfy 
\begin{align}\label{eq.stepsizeconstant}
        \alpha_k = \alpha \leq \tfrac{(1-\hat{\theta})^2}{L_g(1+\hat{\theta})}, \quad \text{and} \quad  \beta_k = \beta \leq \tfrac{\gamma-\gamma_H}{\delta L_H},
    \end{align}
    where $0 \leq \gamma_H < \gamma$.
    If Algorithm \ref{alg.inexact.two-step} terminates finitely in iteration $k \in \Num$, then $\nabla f(x_k) = 0$ and $\lambda_{\min} (\nabla^2 f(x_k)) \geq 0$, i.e., $x_k$ is a second-order stationary point. Otherwise, 
    \begin{equation} \label{eq.inexact.result}
        \begin{split}
            \lim_{k \to \infty}\| \nabla f(x_k)\|_2 = 0 \quad \text{and} \quad \liminf_{k \to \infty} \lambda_{\min} (\nabla^2 f(x_k)) \geq 0.
        \end{split}
    \end{equation}
\etheorem
\bproof
The two-step method terminates finitely if and only if, for some $k \in \Num_{+}$, $p_k = d_k = 0$. In this case, $\lambda_k \geq 0$ and $d_k = 0$, and $x_k = \hat{x}_k$ and $\nabla f(x_k) = \nabla f(\hat{x}_k)$.
    By Condition \ref{cond.inexact.dk} and $d_k = - \hat{g}_k$,
    \begin{equation*}
        \begin{split}
            \|-d_k - \nabla f(\hat{x}_k) \|_2 \leq \hat{\theta} \|d_k\|_2 = 0 \quad \Rightarrow \quad \nabla f(x_k) = \nabla f(\hat{x}_k) = d_k = 0.
        \end{split}
    \end{equation*}
    Similarly, by \eqref{fom.inexact.lbdk}, it follows that $\lambda_{\min, k} \geq 0$. Therefore, $x_k$ is a second-order stationary point. 
    
    If the algorithm does not terminate finitely, consider an arbitrary $k \in \Num$. If $\lambda_k \geq 0$, then $\lambda_k^{-} := \min\{\lambda_k, 0\} = 0$ and $p_k = 0$, and $\hat{x}_k = x_k$ and $f(\hat{x}_k) = f(x_k)$. Otherwise ($\lambda_k < 0$), $p_k \neq 0$ and $\lambda_k^{-} = \lambda_k$. By Assumption \ref{asm.function}, it follows that
    \begin{equation}\label{eq.taylor_negative1}
    \begin{aligned}
        f(x_k+\beta p_k) &\leq f(x_k) + \beta \nabla f(x_k)^{\mathsf{T}} p_k + \tfrac{1}{2} \beta^2 p_k^{\mathsf{T}} \nabla^2 f(x_k) p_k + \tfrac{L_H}{6} \beta^3 \|p_k\|_2^3\\
        & \leq f(x_k) + \tfrac{1}{2} \beta^2 p_k^{\mathsf{T}} \nabla^2 f(x_k) p_k + \tfrac{L_H}{6} \beta^3 \|p_k\|_2^3 \\
        & = f(x_k) + \tfrac{1}{2} \beta^2 p_k^{\mathsf{T}} H_k p_k + \tfrac{1}{2} \beta^2 p_k^{\mathsf{T}} (\nabla^2 f(x_k) - H_k) p_k + \tfrac{L_H}{6} \beta^3 \|p_k\|_2^3  \\
        & \leq f(x_k) + \tfrac{1}{2} \beta^2 \gamma \lambda_k \|p_k\|_2^2 + \tfrac{1}{2} \beta^2 \|p_k\|_2 \left\|(\nabla^2 f(x_k)- H_k) p_k\right\|_2 + \tfrac{L_H}{6} \beta^3 \|p_k\|_2^3  \\
        & \leq f(x_k) - \tfrac{1}{2} \beta^2 \gamma |\lambda_k| \|p_k\|_2^2 + \tfrac{1}{2} \beta^2 \gamma_H |\lambda_k^{-}| \|p_k\|_2^2 + \tfrac{L_H}{6} \beta^3 \|p_k\|_2^3 \\
        & = f(x_k) - \tfrac{1}{2} (\gamma-\gamma_H - \tfrac{L_H}{3} \beta \delta) \beta^2 \delta^2 |\lambda_k|^3 \\
        & \leq f(x_k) - \tfrac{1}{3} (\gamma-\gamma_H) \beta^2 \delta^2 |\lambda_k|^3 
    \end{aligned}
    \end{equation}
    where the second inequality is by \eqref{fom.inexact.pkselect} and \eqref{fom.inexact.gk}, the third inequality is by \eqref{fom.inexact.qk1}, the fourth inequality is 
    by \eqref{fom.inexact.hk}, the equality is by \eqref{fom.inexact.qk2}, and the last inequality is by the negative curvature step constant step size $\beta$ choice. In both cases ($\lambda_k \geq 0$ and $\lambda_k < 0$) inequality  \eqref{eq.taylor_negative1} holds, and for all $k \in \mathbb{N}$, \begin{align}\label{eq.taylor_negative}
        f(\hat{x}_k) = f(x_k + \beta p_k) & \leq f(x_k) - \tfrac{1}{3} \beta^2 (\gamma-\gamma_H) \delta^2 |\lambda_k^-|^3 \nonumber \\
        & \leq f(x_k) - \tfrac{\beta^2 (\gamma-\gamma_H) \delta^2}{3 (1+ \gamma_{\lambda})^3} |\lambda_{\min,k}^-|^3,
    \end{align}
    where $\lambda_k^-$ and $\lambda_{\min,k}^{-}$ are defined in Condition~\ref{cond.inexact.hk}, 
    and the last inequality follows by \eqref{fom.inexact.lbdk} and the fact that when $\lambda_k \geq 0$, $\lambda_k^- = \lambda_{\min,k}^- = 0$.

    With regards to the descent step $x_{k+1} \leftarrow \hat{x}_k + \alpha d_k$, by Assumption~\ref{asm.function}, 
    \begin{align}\label{eq.taylor_descent}
            f(x_{k+1}) & \leq f(\hat{x}_k) + \alpha \nabla f(\hat{x}_k)^{\mathsf{T}} d_k + \alpha^2 \tfrac{L_g}{2}   \|d_k\|_2^2 \nonumber \\
            & \leq f(\hat{x}_k) - \tfrac{\alpha}{1+\hat{\theta}} \|\nabla f(\hat{x}_k)\|_2^2 + \alpha^2\tfrac{L_g }{2(1-\hat{\theta})^2} \|\nabla f(\hat{x}_k)\|_2^2 \nonumber \\
            & = f(\hat{x}_k) - \tfrac{\alpha}{1+\hat{\theta}} \left(1-\tfrac{L_g (1+\hat{\theta})}{2(1-\hat{\theta})^2} \alpha \right) \|\nabla f(\hat{x}_k)\|_2^2 \nonumber \\
            & \leq f(\hat{x}_k) - \tfrac{\alpha}{2(1+\hat{\theta})} \|\nabla f(\hat{x}_k)\|_2^2,
    \end{align}
    where the second inequality is by \eqref{fom.inexact.hatgk}, and the last inequality is by the constant step size choice associated with the descent step \eqref{eq.stepsizeconstant}. 

    Combining inequalities \eqref{eq.taylor_negative} and \eqref{eq.taylor_descent},  \begin{equation}\label{eq.combine}
            f(x_{k+1}) \leq f(x_k) - \tfrac{\beta^2 (\gamma-\gamma_H) \delta^2}{3 (1+ \gamma_{\lambda})^3}  |\lambda_{\min,k}^-|^3 - \tfrac{\alpha}{2(1+\hat{\theta})} \|\nabla f(\hat{x}_k)\|_2^2.
    \end{equation}
     Summing \eqref{eq.combine} from $k = 0$ to $K$, 
    \begin{equation*}
            \tfrac{\beta^2 (\gamma-\gamma_H) \delta^2}{3(1+ \gamma_{\lambda})^3}  \sum_{k=0}^{K} |\lambda_{\min,k}^-|^3 + \tfrac{\alpha}{2(1 + \hat{\theta})} \sum_{k=0}^{K} \|\nabla f(\hat{x}_k)\|_2^2 \leq f(x_0) - f(x_{K+1}) \leq f(x_0) - \bar{f} < \infty.
    \end{equation*}
    Taking the limit $K \rightarrow \infty$, it follows that 
    \begin{align} \label{Eq: limitresult1}
        \sum_{k=0}^{\infty} |\lambda_{\min,k}^-|^3 < \infty \quad \text{ and } \quad \sum_{k=0}^{\infty} \|\nabla f(\hat{x}_k)\|_2^2 < \infty.
    \end{align}
    The latter bound yields
    \begin{equation}\label{eq.limgrad}
        \lim_{k \to \infty} \|\nabla f(\hat{x}_k)\|_2 = 0.
    \end{equation}
    For the former bound \eqref{Eq: limitresult1}, it follows that
    \begin{equation}\label{eq.implies}
        \begin{split}
            \lim_{k \to \infty} |\lambda_{\min,k}^-|^3 = 0 \quad  \Rightarrow \quad \lim_{k \to \infty} \lambda_{\min,k}^- = 0 \quad 
             \Rightarrow \quad \liminf_{k \to \infty} \lambda_{\min}(\nabla^2 f(x_k)) \geq 0.
        \end{split}
    \end{equation}
    The last statement can be proven by contradiction. If $\liminf_{k \to \infty} \lambda_{\min}(\nabla^2 f(x_k)) < 0$, then there exists a constant $c < 0$ such that $\liminf_{k \to \infty} \lambda_{\min}(\nabla^2 f(x_k)) \leq c$. It follows that there is a subsequence of $\{\lambda_{\min}(\nabla^2 f(x_k))\}_{k=1}^{\infty}$ represented as $\{\lambda_{\min}(\nabla^2 f(x_{k_i}))\}_{i=1}^{\infty}$ satisfying $\lambda_{\min}(\nabla^2 f(x_{k_i})) \leq c < 0$ for all $i\in \N{}$. By the definition of $\lambda_{\min,k}^-$, we have $\lambda_{\min,k_i}^- = \lambda_{\min}(\nabla^2 f(x_{k_i})) \leq c < 0$ and $0 = \lim_{i \to \infty} \lambda_{\min,k_i}^- \leq c < 0$ which leads to a contradiction. Thus, the second limit in \eqref{eq.inexact.result} holds.
    
    With regards to the former limit in \eqref{eq.inexact.result}, by 
    \eqref{fom.inexact.qk2} and \eqref{fom.inexact.lbdk},
    \begin{align*}
            \sum_{k=0}^{\infty} \|\hat{x}_k - x_k \|_2^3  = \beta^3 \sum_{k=0}^{\infty} \|p_k \|_2^3 & = \beta^3 \sum_{k=0}^{\infty} \|q_k \|_2^3 \\
            & = \beta^3 \delta^3 \sum_{k=0}^{\infty} |\lambda_k^- |^3 \\
            & \leq \tfrac{\beta^3 \delta^3}{(1-\gamma_{\lambda})^3}  \sum_{k=0}^{\infty} |\lambda_{\min,k}^-|^3 < \infty
        \end{align*}
    from which it follows that $\lim_{k \to \infty} \|\hat{x}_k - x_k \|_2 = 0$. By Assumption~\ref{asm.function} and \eqref{eq.limgrad}, 
    \begin{equation*}
        \begin{split}
            0  \leq \limsup_{k \to \infty} \|\nabla f(x_k)\|_2 & = \limsup_{k \to \infty} \|\nabla f(x_k)- \nabla f(\hat{x}_k) + \nabla f(\hat{x}_k)\|_2 \\
            & \leq \limsup_{k \to \infty} \|\nabla f(x_k)- \nabla f(\hat{x}_k)\|_2 + \limsup_{k \to \infty} \|\nabla f(\hat{x}_k)\|_2 \\
            & \leq L_g \limsup_{k \to \infty} \| x_k- \hat{x}_k\|_2 + \limsup_{k \to \infty} \|\nabla f(\hat{x}_k)\|_2 = 0
        \end{split}
    \end{equation*}
    which implies the first limit in \eqref{eq.inexact.result}.
\eproof

Theorem~\ref{thm.inexact.converge} shows that the iterates generated by Algorithm~\ref{alg.inexact.two-step} converge to second-order stationary points. This is expected under the stated assumptions and conditions, and matches the convergence results of the deterministic two-step method up to constants related to the gradient and Hessian approximations, and the two search directions. 
As a corollary to Theorem~\ref{thm.inexact.converge}, we derive the following iteration complexity result.
\begin{corollary}\label{cor.complex1}
Consider any scalars $\epsilon_g > 0$ and $\epsilon_H>0$. With respect to Algorithm \ref{alg.inexact.two-step}, the cardinality of the index set $\mathcal{G}(\epsilon_g) := \{k \in \mathbb{N}: \|\nabla f({x}_{k}) \|_2 > \epsilon_g \}$ 
is at most $\mathcal{O}(\epsilon_g^{-2})$, and the cardinality of the index set $\mathcal{H}(\epsilon_H) := \{k \in \mathbb{N}: |\lambda_{\min,k}^-| > \epsilon_H \}$
is at most $\mathcal{O}(\epsilon_H^{-3})$. Hence, the number of iterations and derivative (i.e., gradient and Hessian) evaluations required until iteration $k \in \mathbb{N}$ is reached with $\|\nabla f(x_k) \|_2 \leq \epsilon_g$ and $\lambda_{\min}(\nabla^2 f(x_k)) \geq -\epsilon_H$ 
is at most $\mathcal{O}(\max\{\epsilon_g^{-2}, \epsilon_H^{-3}\})$. 
\end{corollary}
\begin{proof}
    By Assumption \ref{asm.function}, the iterate update $x_{k+1} \leftarrow \hat{x}_{k} + \alpha d_k$ and Condition~\ref{cond.inexact.dk}, 
    \begin{equation*}
        \begin{split}
            \|\nabla f(x_{k+1})\|_2 & = \|\nabla f(x_{k+1}) - \nabla f(\hat{x}_k) + \nabla f(\hat{x}_k)\|_2 \\
            & \leq L_g \alpha \|d_k\|_2  + \| \nabla f(\hat{x}_k)\|_2\\
            & \leq \tfrac{L_g \alpha}{1-\hat{\theta}}\| \nabla f(\hat{x}_k)\|_2 +  \|\nabla f(\hat{x}_k)\|_2 = \left(1+\tfrac{L_g \alpha}{1-\hat{\theta}} \right) \| \nabla f(\hat{x}_k)\|_2.
        \end{split}
    \end{equation*}
    Thus, given $\epsilon_g \geq 0$, it follows that the set \begin{align*}
        \mathcal{G}(\epsilon_g) = \left\{k \in \N{}_{+}: \|\nabla f(x_{k})\|_2 > \epsilon_g\right\} \subseteq \left\{ k \in \N{}_{+}: \|\nabla f(\hat{x}_{k-1})\|_2 > \left(1+\tfrac{L_g \alpha}{1-\hat{\theta}} \right)^{-1} \epsilon_g \right\} := \hat{\mathcal{G}}(\epsilon_g).
    \end{align*}
    For all $k \in \mathbb{N}$, \eqref{eq.combine} holds, from which it follows that, 
    \begin{align*}
             k \in \mathcal{G}(\epsilon_g) \subseteq \hat{\mathcal{G}}(\epsilon_g) \quad  &\Longrightarrow \quad f(x_{k-1}) - f(x_{k}) \geq \tfrac{\alpha}{2 (1+\hat{\theta})} \left(1+\tfrac{L_g \alpha}{1-\hat{\theta}} \right)^{-2} \epsilon_g^2 = \tfrac{(1-\hat{\theta})^2}{8 L_g} \epsilon_g^2, \\
            \text{and } \quad 
            k \in \mathcal{H}(\epsilon_H) \quad  &\Longrightarrow \quad f(x_k) - f(x_{k+1}) \geq \tfrac{\beta^2 (\gamma-\gamma_H) \delta^2}{3 (1+ \gamma_{\lambda})^3} \epsilon_H^3 = \tfrac{(\gamma-\gamma_H)^3}{3 L_H^2 (1+ \gamma_{\lambda})^3} \epsilon_H^3.
        \end{align*}
    Since $f$ is bounded below by $\bar{f}$ and \eqref{eq.combine} ensures that $\{f(x_k)\}$ monotonically decreases, the inequalities above imply that $\mathcal{G}(\epsilon_g)$ and $\mathcal{H}(\epsilon_H)$ are both finite sets. In addition, by summing the reductions achieved in $f$ over the iterations, it follows that
    \begin{equation*}
        \begin{split}
            & f(x_0) - \bar{f} \geq \sum_{k \in \mathcal{G}(\epsilon_g)} (f(x_{k-1}) - f(x_{k})) \geq |\mathcal{G}(\epsilon_g)| \tfrac{(1-\hat{\theta})^2}{8L_g} \epsilon_g^2, \\
            \text{and } \quad & f(x_0) - \bar{f} \geq \sum_{k \in \mathcal{H}(\epsilon_H)} (f(x_k) - f(x_{k+1})) \geq |\mathcal{H}(\epsilon_H)| \tfrac{(\gamma-\gamma_H)^3}{3 L_H^2 (1+ \gamma_{\lambda})^3} \epsilon_H^3.
        \end{split}
    \end{equation*}
    Rearranging, completes the proof with 
    \begin{equation*}
        |\mathcal{G}(\epsilon_g)| \leq \tfrac{8L_g(f(x_0)-\bar{f})}{(1-\hat{\theta})^2} \epsilon_g^{-2} \quad \text{and} \quad |\mathcal{H}(\epsilon_H)| \leq \tfrac{3 L_H^2 (1+ \gamma_{\lambda})^3 (f(x_0) -  \bar{f})}{(\gamma-\gamma_H)^3} \epsilon_H^{-3}.
    \end{equation*}
\end{proof}
Similar to Theorem~\ref{thm.inexact.converge}, the result in Corollary~\ref{cor.complex1} matches that of the deterministic analog in terms of the constants $\epsilon_g >0$ and $\epsilon_H > 0 $ \cite{curtis2019exploiting, royer2018complexity}.

\section{Stochastic Setting} \label{sec.stoch_setting}

In this section, we focus on stochastic unconstrained optimization problems, i.e., 
\begin{equation*}
    \min_{x\in\R{n}} f(x) = \E_{\xi}[F(x, \xi)],
\end{equation*}
where $\xi$ is a random variable with associated probability space $(\Xi, \mathcal{F}, \Prob)$ and $F: \R{n} \times \Xi \to \R{}$. We assume access only to stochastic approximations of the objective function, the gradient and the Hessian. As in Section~\ref{sec.inexact}, we consider a two-step method, we present conditions on the two directions ($p_k$ and $d_k$) in expectation (Section~\ref{subsec.stoch_setting.conds}), and we derive convergence and complexity guarantees in expectation under different step size choices (Section~\ref{subsec.stoch_setting.theory}).

\subsection{Assumptions, Conditions, and Algorithm} \label{subsec.stoch_setting.conds}

Throughout this section, Assumption~\ref{asm.function} holds. Similar to Section~\ref{sec.inexact}, we impose conditions on the gradient and Hessian approximations computed and the search directions employed. In this section, these conditions are in expectation. To this end, we introduce the following conditional expectations. We define $\E[\cdot | \mathcal{F}_{k}]$ as $\E_{k}[\cdot]$ where $\mathcal{F}_{k} = \sigma(x_0, \dots, x_{k})$ is the $\sigma$-algebra generated by $x_0$, $\dots$, $x_{k}$, i.e., the history of the algorithm up to iteration $k$. Similarly, $\E[\cdot | \mathcal{F}_{k}, \hat{x}_k]$ is defined as $\E_{\hat{k}}[\cdot]$ and $\E[\cdot | \mathcal{F}_{k}, q_k]$ is defined as $\E_{k,q}[\cdot]$.

\bcondition \label{cond.sto.pk} For all $k\in \mathbb{N}$, if $\lambda_k \geq 0$ (the left-most eigenvalue of $H_k$), the negative
curvature direction is $p_k \leftarrow 0$. Otherwise, it is chosen by 
\begin{align*}
    p_k = \begin{cases}
        q_k, & g_k^{\mathsf{T}}q_k \leq 0 \\
        -q_k, & \text{otherwise}
    \end{cases}
\end{align*}
where $q_k$ is a negative curvature vector (with respect to the matrix $H_k$) such that
\begin{equation}\label{fom.sto.qk}
    \begin{split}
        q_k^{\mathsf{T}} H_k q_k &\leq \gamma \lambda_k \|q_k\|_2^2 < 0, \quad \gamma \in (0, 1] \\
        \|q_k\|_2 &= \delta |\lambda_k|, \quad\quad \quad \quad\delta \in (0, \infty)
    \end{split}
\end{equation}
and $g_k$ is the approximation of $\nabla f(x_k)$ satisfying
\begin{equation} \label{fom.sto.gk}
    \E_{k, q} \left[|(g_k - \nabla f(x_k))^{\mathsf{T}} q_k| \right] \leq \theta_k^2 \E_{k, q} \left[ |g_k^{\mathsf{T}} q_k| \right] + \sigma_k^2, \quad \theta_k \in [0, 1), \ \sigma_k \in [0, \infty) .
\end{equation}
\econdition
\bremark
Given a Hessian approximation $H_k$, the computation of the vector $q_k$ is the same as the deterministic analogue \eqref{fom.inexact.qk}. As compared to the deterministic gradient condition \eqref{fom.inexact.gk}, \eqref{fom.sto.gk} holds in expectation (conditioned on $\mathcal{F}_k$ and $q_k$) and has an additional term $\sigma_k^2$ that captures possibly non-diminishable errors in the approximation of $g_k^{\mathsf{T}}q_k$.
\eremark

Similar to the deterministic setting, a stochastic version of Condition \ref{cond.inexact.hk} is required for the stochastic Hessian approximation $H_k$. 

\bcondition\label{cond.sto.hk}
    For all $k \in \N{}$, the stochastic Hessian approximation $H_k$ and negative curvature direction $p_k$ satisfy in expectation
    \begin{equation}\label{fom.sto.hk}
        \begin{split}
            \E_{k}\left[\|(H_k - \nabla^2 f(x_k)) p_k\|_2^2 \right] \leq \gamma_H^2 |\lambda_{\min,k}^- |^2  \E_{k} \left[\|p_k\|_2^2 \right], \quad \gamma_H \in [0, \gamma),
        \end{split}
    \end{equation}
    where $\lambda_{\min, k}^- := \min\{\lambda_{\min} \left(\nabla^2 f(x_k) \right), 0 \}$. 
    Moreover, the gap between the left-most eigenvalues of $H_k$ and $\nabla^2 f(x_k)$ is bounded by \begin{equation}\label{fom.sto.lbdk}
        \begin{split}
            \E_{k}\left[ | \lambda_k^- - \lambda_{\min, k}^- | \right] \leq \gamma_{\lambda} |\lambda_{\min, k}^-|, \quad \gamma_{\lambda} \in [0, 1)
        \end{split}
    \end{equation}
    where 
    $\lambda_k^- := \min\{\lambda_k, 0\}$.
\econdition
\bremark
When $\lambda_{\min} \left(\nabla^2 f(x_k) \right) \geq 0$, by definition $\lambda_{\min, k}^- = 0$, and inequality \eqref{fom.sto.lbdk} implies  $\lambda_k^- = 0$ with probability $1$. From this, it follows that $\lambda_k \geq 0$ and $p_k = 0$ with probability $1$ which in turn justifies the right-hand-side of \eqref{fom.sto.hk}.
To justify the use of $q_k$ in the computation of the negative curvature direction $p_k$ and to verify the conditions thereof, we first note that, $p_k = \mathds{1}(g_k^{\mathsf{T}}q_k \leq 0)q_k + \mathds{1}(g_k^{\mathsf{T}}q_k > 0)(-q_k)$ and $\mathds{1}(g_k^Tq_k \leq 0)+\mathds{1}(g_k^Tq_k > 0) = 1$. It follows that 
\begin{align*}
    \E_k[\|p_k\|_2^2] & = \E_k \left[\|\mathds{1}(g_k^{\mathsf{T}}q_k \leq 0)q_k + \mathds{1}(g_k^{\mathsf{T}}q_k > 0)(-q_k)\|_2^2 \right] \\
    & = \E_k[\mathds{1}(g_k^{\mathsf{T}}q_k \leq 0)\|q_k\|_2^2 + \mathds{1}(g_k^{\mathsf{T}}q_k > 0) \|-q_k\|_2^2] \\
    & = \E_k[(\mathds{1}(g_k^{\mathsf{T}}q_k \leq 0)+ \mathds{1}(g_k^{\mathsf{T}}q_k > 0)) \|q_k\|_2^2] = \E_k[\|q_k\|_2^2].
\end{align*}
This idea also applies to other inequalities involving norms of $q_k$ and $p_k$.
\eremark

Given a direction of negative curvature, the iterate is updated via $\hat{x}_{k} \leftarrow x_k + \beta_k p_k$, where $\beta_k > 0$ is the step size, and the algorithm proceeds to the gradient-related step $d_k$. For the direction $d_k$ at the point $\hat{x}_k$, we assume $d_k \leftarrow -\hat{g}_k$ being a negative gradient approximation and $\hat{g}_k$ has the following condition.
\bcondition\label{cond.sto.dk}
For all $k \in \N{}$, the stochastic gradient approximation $\hat{g}_k$ satisfies 
\begin{subequations}\label{fom.sto.hatgk}
\begin{align}
        \E_{\hat{k}} \left[\hat{g}_k \right] & = \nabla f(\hat{x}_k), \label{fom.sto.hatgk1}\\
        \E_{\hat{k}} \left[ \|\hat{g}_k - \nabla f(\hat{x}_k)\|_2^2 \right] & \leq \hat{\theta}_k^2 \|\nabla f(\hat{x}_k)\|_2^2 + \hat{\sigma}_k^2, \quad \hat{\theta}_k \in [0, \infty), \ \hat{\sigma}_k \in [0, \infty). \label{fom.sto.hatgk2}
    \end{align}
    \end{subequations}
\econdition

\bremark
Condition \ref{cond.sto.dk} is commonly used in the context of stochastic gradient-based methods \cite{bottou2018optimization}. This condition can be generalized; see e.g., \cite[Assumption 4.3]{bottou2018optimization}. The range of the parameter $\hat{\theta}_k$ in the stochastic setting is less restricted than its variant in the deterministic setting \eqref{fom.inexact.hatgk}. This is because our goal in the deterministic setting is to prove a stronger result which requires a stronger assumption, whereas such a result is not appropriate in the stochastic setting.
\eremark

One may notice that in the conditions for the deterministic setting (Conditions~\ref{cond.inexact.pk}, \ref{cond.inexact.hk} and \ref{cond.inexact.dk}), we use inexact information on the right-hand-side, such as $\lambda_k$ and $\|\hat{g}_k\|_2$, while in the conditions for the stochastic setting (Conditions~\ref{cond.sto.pk}, \ref{cond.sto.hk} and \ref{cond.sto.dk}), we use instead the exact information. The reason for this disparity is that in the deterministic setting, both the exact and inexact information based conditions are mathematically equivalent up to some constants. However, in the stochastic setting, such equivalence is not possible due to the randomness in the quantities computed, i.e., $g_k$, $\lambda_k$, $H_k$, and the search directions. Since exact information is necessary to ensure convergence to the solution, we use the exact information on the right-hand side of the conditions in the stochastic setting.

The complete stochastic two-step method is presented in 
Algorithm \ref{alg.sto.two-step}.
\begin{algorithm}[h]
    \caption{Stochastic Two-Step Method}\label{alg.sto.two-step}
    \begin{algorithmic}[1]
        \Require{$x_{0} \in \Real^n$, $\alpha_k > 0$, $\beta_k > 0$}
        \ForAll{$k \in \{0, 1, \dots\}$} 
        \State compute Hessian approximation $H_k$ that satisfies Condition \ref{cond.sto.hk}    
            \If{$\lambda_k \geq 0$} \label{line.sto.term_negative}
                \State set $p_k \leftarrow 0$ 
                \Else \State compute $q_k$ and set $p_k$ using Condition \ref{cond.sto.pk}
            \EndIf
            \State set $\hat{x}_k \leftarrow x_k + \beta_k p_k$
            \State compute gradient approximation $\hat{g}_k$ that satisfies Condition \ref{cond.sto.dk}  
            \If{$\hat{g}_k = 0$} \label{line.sto.term_descent}
                \State set $d_k \leftarrow 0$ 
            \Else \State set $d_k \leftarrow - \hat{g}_k$
            \EndIf
            \State set $x_{k+1} \leftarrow \hat{x}_k + \alpha_k d_k = x_k + \beta_k p_k + \alpha_k d_k$
		\EndFor
    \end{algorithmic}
\end{algorithm}

\bremark 
We note that Algorithm~\ref{alg.sto.two-step} has the same structure as Algorithm~\ref{alg.inexact.two-step} (deterministic inexact setting). One key difference is the fact that the stochastic algorithm does not have an explicit termination condition (analogous to Line~\ref{line.terminate} in  Algorithm~\ref{alg.inexact.two-step}) due to the stochasticity. Similar to the stochastic gradient method \cite{bottou2018optimization}, the step size sequences for the two steps in Algorithm~\ref{alg.sto.two-step} can be set as constant (sufficiently small) or adaptive (e.g., diminishing). The specific ranges for these two cases depend on user and problem-specific parameters, and the exact forms are given in Theorems \ref{thm.sto.conv.dimstep} and \ref{thm.sto.conv.dimvar}, respectively. 
\eremark

\subsection{Convergence and Complexity Results} \label{subsec.stoch_setting.theory}

In this subsection, we present the theoretical results for Algorithm \ref{alg.sto.two-step} under two different scenarios: $(i)$ constant variance with diminishing step sizes and $(ii)$ diminishing variance with constant step size.
Before showing the main convergence results, we first introduce a fundamental lemma. 
\blemma \label{lmm: basineq}
    Suppose Assumption \ref{asm.function} and 
    Conditions \ref{cond.sto.pk}, \ref{cond.sto.hk}, and \ref{cond.sto.dk} hold 
    with parameters $\gamma \in (0,1]$, $\gamma_H \in [0,\gamma)$, $\gamma_{\lambda} \in [0,1)$, $\gamma_H < \gamma (1-\gamma_{\lambda})$, $\delta \in (0, \infty)$, and $\hat{\theta}_k \in [0,\infty)$. Let the step size parameters satisfy 
    \begin{align}\label{eq.stepsizeconstant_sto}
        0< \alpha_k \leq \tfrac{1}{L_g \left(1+\hat{\theta}_k^2\right)} \quad \text{and} \quad 0 < \beta_k \leq \tfrac{3\left(\gamma - \tfrac{\gamma_H}{1-\gamma_{\lambda}}\right)}{2 \delta L_H}.
    \end{align}
    Then, for $k \in \N{}$ 
    \begin{equation}\label{important_eq}
        \begin{aligned}
            \E_{k}[f(x_{k+1})]  \leq f(x_k) &- \beta_k^2 \tfrac{\delta^2 }{4}  \left(\gamma - \tfrac{\gamma_H}{1-\gamma_{\lambda}}\right) (1-\gamma_{\lambda})^3 |\lambda_{\min, k}^-|^3 \\
            & - \tfrac{\alpha_k}{2} \E_{k} [\|\nabla f(\hat{x}_k) \|_2^2] + \alpha_k^2 \hat{\sigma}_k^2\tfrac{L_g}{2} + \beta_k \sigma_k^2.
        \end{aligned}
    \end{equation}
\elemma
\bproof
At a given iterate $x_k$, if $\lambda_k \geq 0$, then $p_k = 0$ and $\hat{x}_k = x_k$, which means $f(\hat{x}_k) = f(x_k)$. Otherwise, $p_k \neq 0$, $\lambda_k^- = \lambda_k$ and it follows that
    \begin{align} \label{Inq:thm5-1}
            f(x_k + \beta_k p_k) & \leq f(x_k) + \beta_k \nabla f(x_k)^{\mathsf{T}} p_k + \tfrac{1}{2} \beta_k^2 p_k^{\mathsf{T}} \nabla^2 f(x_k) p_k + \tfrac{L_H}{6}\beta_k^3 \|p_k\|_2^3 \nonumber \\
            & = f(x_k) + \beta_k (\nabla f(x_k)-g_k)^{\mathsf{T}} p_k + \beta_k g_k^{\mathsf{T}} p_k + \tfrac{1}{2} \beta_k^2 p_k^{\mathsf{T}} H_k p_k \nonumber \\
            & \quad \quad + \tfrac{1}{2} \beta_k^2 p_k^{\mathsf{T}} (\nabla^2 f(x_k) - H_k) p_k + \tfrac{L_H}{6} \beta_k^3 \|p_k\|_2^3 \nonumber \\
            & \leq f(x_k) + \beta_k |(\nabla f(x_k)-g_k)^{\mathsf{T}}p_k | + \beta_k g_k^{\mathsf{T}} p_k + \tfrac{1}{2} \gamma \beta_k^2 \lambda_k \|p_k\|_2^2 \nonumber \\
            & \quad \quad + \tfrac{1}{2} \beta_k^2 \|p_k\|_2 \left\|(\nabla^2 f(x_k)- H_k) p_k\right\|_2 + \tfrac{L_H}{6} \beta_k^3 \|p_k\|_2^3 \nonumber \\
            & = f(x_k) + \beta_k |(\nabla f(x_k)-g_k)^{\mathsf{T}}p_k | + \beta_k g_k^{\mathsf{T}} p_k - \tfrac{1}{2} \gamma \beta_k^2 \delta^2 |\lambda_k|^3 \nonumber \\
            & \quad \quad + \tfrac{1}{2} \beta_k^2 \delta |\lambda_k| \left\|(\nabla^2 f(x_k)- H_k) p_k\right\|_2 + \tfrac{L_H}{6} \beta_k^3 \delta^3 |\lambda_k|^3 \nonumber \\
            & = f(x_k) + \beta_k |(\nabla f(x_k)-g_k)^{\mathsf{T}}p_k | + \beta_k g_k^{\mathsf{T}} p_k - \tfrac{1}{2} \gamma \beta_k^2 \delta^2 |\lambda_k^-|^3 \\
            & \quad \quad + \tfrac{1}{2} \beta_k^2 \delta |\lambda_k^-| \left\|(\nabla^2 f(x_k)- H_k) p_k\right\|_2 + \tfrac{L_H}{6} \beta_k^3 \delta^3 |\lambda_k^-|^3, \nonumber
    \end{align}
    where in the second inequality we use Condition \ref{cond.sto.pk}. Note that, the above inequality also holds when $\lambda_k \geq 0$ in which case $\lambda_k^- = 0$ and $p_k = 0$.
    By \eqref{fom.sto.gk}, it follows that
    \begin{align} \label{fom.sto.pk_descent}
            \E_{k} [ |(\nabla f(x_k)-g_k)^{\mathsf{T}}p_k |] + \E_{k} [g_k^{\mathsf{T}} p_k]
            & = \E_{k}\left[ \E_{k, q} [|(g_k - \nabla f(x_k))^{\mathsf{T}} p_k| ] \right] + \E_{k} [g_k^{\mathsf{T}} p_k] \nonumber \\
            &= \E_{k}\left[ \E_{k, q} [|(g_k - \nabla f(x_k))^{\mathsf{T}} q_k| ] \right] + \E_{k} [g_k^{\mathsf{T}} p_k] \nonumber \\
            & \leq \E_{k}\left[ \theta_k^2 \E_{k, q} [ |g_k^{\mathsf{T}} q_k| ] + \sigma_k^2 \right] + \E_{k} [g_k^{\mathsf{T}} p_k] \nonumber \\
            & = \theta_k^2 \E_{k} [ |g_k^{\mathsf{T}} p_k| ] + \E_{k} [g_k^{\mathsf{T}} p_k]  + \sigma_k^2 \nonumber \\
            & = (1-\theta_k^2) \E_{k} [g_k^{\mathsf{T}} p_k] + \sigma_k^2 \leq \sigma_k^2.
    \end{align}
    Then, taking the expectation on the both sides of \eqref{Inq:thm5-1} with respect to $x_k$, it follows that 
    \begin{align}\label{eq.inequality}
        \E_{k}[f(x_k + \beta_k p_k)] & \leq f(x_k) + \beta_k \E_{k}[ |(\nabla f(x_k)-g_k)^{\mathsf{T}}p_k |] + \beta_k \E_{k} [g_k^{\mathsf{T}} p_k] - \tfrac{1}{2} \gamma \beta_k^2 \delta^2 \E_{k}[|\lambda_k^-|^3] \nonumber \\
        & \quad \quad + \tfrac{1}{2} \beta_k^2 \delta \E_{k} \left[ |\lambda_k^-| \left\|(\nabla^2 f(x_k)- H_k) p_k \right\|_2 \right] + \tfrac{L_H}{6} \beta_k^3 \delta^3 \E_{k}[|\lambda_k^-|^3] \nonumber \\
        & \leq f(x_k) + \beta_k \sigma_k^2 - \tfrac{\gamma}{2} \beta_k^2 \delta^2 \E_{k}[|\lambda_k^-|^3] + \tfrac{L_H}{6} \beta_k^3 \delta^3 \E_{k}[|\lambda_k^-|^3] \nonumber \\
        & \quad \quad + \tfrac{1}{2} \beta_k^2 \delta \sqrt{\E_{k}[|\lambda_k^-|^2]} \sqrt{\E_{k} \left[\left\|(\nabla^2 f(x_k)- H_k) p_k \right\|_2^2 \right]} \nonumber \\
        & \leq f(x_k) - \tfrac{\gamma}{2} \beta_k^2 \delta^2 \E_{k}[|\lambda_k^-|^3] + \tfrac{L_H}{6} \beta_k^3 \delta^3 \E_{k}[|\lambda_k^-|^3] \nonumber \\
        & \quad \quad + \tfrac{1}{2} \beta_k^2 \delta \gamma_H |\lambda_{\min, k}^{-}|\sqrt{\E_{k}[|\lambda_k^-|^2]}  \sqrt{\E_{k} [\|p_k\|_2^2]} + \beta_k \sigma_k^2 \nonumber \\
        & \leq f(x_k) - \tfrac{\gamma}{2} \beta_k^2 \delta^2 \E_{k}[|\lambda_k^-|^3] + \tfrac{L_H}{6} \beta_k^3 \delta^3 \E_{k}[|\lambda_k^-|^3] + \tfrac{\gamma_H}{2(1-\gamma_{\lambda})} \beta_k^2 \delta^2 (\E_{k}[|\lambda_k^-|^2])^{3/2} \nonumber \\
        & \quad \quad  + \beta_k \sigma_k^2 \nonumber \\
        & \leq f(x_k) - \tfrac{\gamma}{2} \beta_k^2 \delta^2 \E_{k}[|\lambda_k^-|^3] + \tfrac{L_H}{6} \beta_k^3 \delta^3 \E_{k}[|\lambda_k^-|^3] + \tfrac{\gamma_H}{2(1-\gamma_{\lambda})} \beta_k^2 \delta^2 \E_{k}[|\lambda_k^-|^3] \nonumber \\
        & \quad \quad  + \beta_k \sigma_k^2 \nonumber \\
        & = f(x_k) - \tfrac{1}{2} \beta_k^2 \delta^2 \left(\gamma - \tfrac{\gamma_H}{1-\gamma_{\lambda}} - \tfrac{L_H}{3} \beta_k \delta \right) \E_{k}[|\lambda_k^-|^3] + \beta_k \sigma_k^2 \nonumber \\
        & \leq f(x_k) - \tfrac{1}{4} \beta_k^2 \delta^2 \left(\gamma - \tfrac{\gamma_H}{1-\gamma_{\lambda}}\right) \E_{k}[|\lambda_k^-|^3] + \beta_k \sigma_k^2 \nonumber \\
        & \leq f(x_k) - \tfrac{1}{4} \beta_k^2 \delta^2 \left(\gamma - \tfrac{\gamma_H}{1-\gamma_{\lambda}}\right) (1-\gamma_{\lambda})^3 |\lambda_{\min, k}^-|^3 + \beta_k \sigma_k^2.
    \end{align}
    The second inequality follows by \eqref{fom.sto.pk_descent} and the third inequality by \eqref{fom.sto.hk}. The fourth and last inequalities follow by a property induced by \eqref{fom.sto.lbdk}, that is, $
        \tfrac{1}{1+\gamma_{\lambda}} \E_{k}[|\lambda_k^-|] \leq |\lambda_{\min, k}^-| \leq \tfrac{1}{1-\gamma_{\lambda}} \E_{k}[|\lambda_k^-|]$. The fifth inequality follows by $
        \left(\E_{k}[|\lambda_k^-|^2] \right)^{1/2} \leq \left(\E_k[|\lambda_{k}^-|^3] \right)^{1/3}$ (since $|x|^2$ and $|x|^{3/2}$ are convex, this inequality can be proven by applying Jensen's inequality to $|x|^2$ and $|x|^{3/2}$, i.e., for any random variable $X$ with $\E[|X|] < \infty$, we have
    $
        (\E[|X|])^3 \leq \left(\E[|X|^2] \right)^{3/2} \leq \E[|X|^3]
    $.).  
    When $\lambda_k \geq 0$, by definition $\lambda_k^- = 0$ and 
    $p_k = 0$, so \eqref{eq.inequality} holds. Thus, for $k \in \N{}$, 
    \begin{align*}
        \E_{k}[f(\hat{x}_k)] = \E_{k}[f(x_k + \beta_k p_k)] \leq f(x_k) - \tfrac{1}{4} \beta_k^2 \delta^2 \left(\gamma - \tfrac{\gamma_H}{1-\gamma_{\lambda}}\right) (1-\gamma_{\lambda})^3 |\lambda_{\min, k}^-|^3 + \beta_k \sigma_k^2.
    \end{align*}
    
    We derive a similar inequality for the gradient-type step. By Assumption~\ref{asm.function} and the update step, i.e., 
    $x_{k+1} \leftarrow \hat{x}_k + \alpha_k d_k$, it follows that 
    \begin{align}\label{eq.inequality_gradient}
            f(x_{k+1}) & \leq f(\hat{x}_k) + \alpha_k \nabla f(\hat{x}_k)^{\mathsf{T}} d_k + \alpha_k^2\tfrac{L_g}{2}  \|d_k\|_2^2 \nonumber \\
            & = f(\hat{x}_k) - \alpha_k \nabla f(\hat{x}_k)^{\mathsf{T}} \hat{g}_k + \alpha_k^2\tfrac{L_g}{2}  \|\hat{g}_k\|_2^2.
    \end{align}
    Taking the expectation of \eqref{eq.inequality_gradient} conditioned on the fact that the algorithm has reached the iterate $\hat{x}_k$, by Condition~\ref{cond.sto.dk} and the step size condition \eqref{eq.stepsizeconstant_sto} 
    \begin{align*}
            \E_{\hat{k}}[f(x_{k+1})] & \leq f(\hat{x}_k) - \alpha_k \E_{\hat{k}} [\nabla f(\hat{x}_k)^{\mathsf{T}} \hat{g}_k] + \alpha_k^2\tfrac{L_g}{2}  \E_{\hat{k}} [\|\hat{g}_k\|_2^2] \\
            & \leq f(\hat{x}_k) - \alpha_k \|\nabla f(\hat{x}_k)\|_2^2 + \alpha_k^2\tfrac{L_g}{2}  \E_{\hat{k}} [\|\hat{g}_k\|_2^2] \\
            & \leq f(\hat{x}_k) - \alpha_k \|\nabla f(\hat{x}_k)\|_2^2 + \alpha_k^2\tfrac{L_g}{2}  \left(\hat{\sigma}_k^2 + (1+ \hat{\theta}_k^2) \|\nabla f(\hat{x}_k)\|_2^2\right) \\
            & = f(\hat{x}_k) - \alpha_k \left(1 - \tfrac{L_g}{2} (1+ \hat{\theta}_k^2) \alpha_k \right) \|\nabla f(\hat{x}_k)\|_2^2 + \tfrac{L_g}{2} \alpha_k^2 \hat{\sigma}_k^2 \\
            & \leq f(\hat{x}_k) - \tfrac{\alpha_k}{2} \|\nabla f(\hat{x}_k)\|_2^2 + \tfrac{L_g}{2} \alpha_k^2 \hat{\sigma}_k^2.
        \end{align*}
    Taking the expectation again conditioned on the fact that the algorithm has reached the iterate $x_k$, the inequality becomes
    \begin{align*}
            \E_{k}[f(x_{k+1})] & \leq \E_{k}[f(\hat{x}_k)] - \tfrac{\alpha_k}{2} \E_{k} [\|\nabla f(\hat{x}_k) \|_2^2] + \tfrac{L_g}{2} \alpha_k^2 \hat{\sigma}_k^2 \\
            & \leq f(x_k) - \beta_k^2\tfrac{\delta^2}{4}   (\gamma - \tfrac{\gamma_H}{1-\gamma_{\lambda}}) (1-\gamma_{\lambda})^3 |\lambda_{\min, k}^-|^3 \\
            & \quad \quad \quad \quad - \tfrac{\alpha_k}{2} \E_{k} [\|\nabla f(\hat{x}_k) \|_2^2] + \alpha_k^2 \hat{\sigma}_k^2\tfrac{L_g}{2}  + \beta_k \sigma_k^2
        \end{align*}
    which completes the proof.
\eproof
\bremark
    As compared to the update inequality in the deterministic inexact setting \eqref{eq.combine}, in the stochastic setting, and under Conditions \ref{cond.sto.pk}, \ref{cond.sto.hk}, and \ref{cond.sto.dk}, the update inequality in expectation \eqref{important_eq} has two additional positive terms, $\alpha_k^2 \hat{\sigma}_k^2\tfrac{L_g}{2}$ and $\beta_k \sigma_k^2$. The former comes from the variance of $\hat{g}_k$ and the latter comes from $\nabla f(x_k)^{\mathsf{T}} q_k$ which may not be negative in expectation.
\eremark

We consider two different step size selection strategies and error settings: $(i)$ constant variance and diminishing step size (Theorem~\ref{thm.sto.conv.dimstep}); and $(ii)$ diminishing variance and constant step size (Theorem~\ref{thm.sto.conv.dimvar}). In the former setting, we are only able to show a convergence result similar to those of the classical stochastic gradient method \cite[Theorem 4.9]{bottou2018optimization} and cannot prove convergence to second-order stationary points. In the latter setting, we prove a stronger result and show second-order convergence in expectation.

\btheorem
\label{thm.sto.conv.dimstep}
    Suppose Assumption \ref{asm.function} and Conditions \ref{cond.sto.pk}, \ref{cond.sto.hk} and \ref{cond.sto.dk} hold with parameters $\gamma \in (0,1]$, $\gamma_H \in [0,\gamma)$, $\gamma_{\lambda} \in [0,1)$, $\gamma_H < \gamma (1-\gamma_{\lambda})$, $\delta \in (0, \infty)$, $\theta_k \in [0, 1)$, $\hat{\theta}_k \in [0, \infty)$, $\sigma_k = \sigma$ and $\hat{\sigma}_k = \hat{\sigma} > 0$. Let the step size parameters satisfy \eqref{eq.stepsizeconstant_sto} and 
    \begin{align}\label{stepsize_eq}
        \sum_{k=0}^{\infty} \alpha_k = \infty, \quad \sum_{k=0}^{\infty} \alpha_k^2 < \infty, \quad \text{and} \quad \sum_{k=0}^{\infty} \beta_k < \infty.
    \end{align}
    Then, the iterates generated by Algorithm \ref{alg.sto.two-step} satisfy
    \begin{equation*}
        \begin{split}
            \liminf_{k \to \infty} \E \left[\| \nabla f(x_k)\|_2 \right] = 0.
        \end{split}
    \end{equation*} 
\etheorem
\begin{proof}
    By Lemma \ref{lmm: basineq}, 
    \begin{equation*}
        \begin{split}
            \E_k[f(x_{k+1})] 
            & \leq f(x_k) -  \beta_k^2 c_\delta |\lambda_{\min, k}^{-}|^3 - \tfrac{\alpha_k}{2} \E_{k} [\|\nabla f(\hat{x}_k) \|_2^2] + \tfrac{L_g}{2} \alpha_k^2 \hat{\sigma}^2 + \beta_k \sigma^2,
        \end{split}
    \end{equation*}
    where $c_\delta:=  \tfrac{\delta^2}{4}(\gamma - \tfrac{\gamma_H}{1-\gamma_{\lambda}}) (1-\gamma_{\lambda})^3$. Taking the total expectation over all random variables up to iteration $k$ starting with $x_0$, and summing the above from 
    $k=0$ to $K$,
    \begin{equation*}
        \begin{split}
        \E[f(x_{K+1})] & \leq f(x_0) - c_\delta\sum_{k=0}^{K} \beta_k^2 \E[|\lambda_{\min, k}^-|^3] - \tfrac{1}{2}\sum_{k=0}^{K} \alpha_k \E[\|\nabla f(\hat{x}_k) \|_2^2] + \tfrac{L_g \hat{\sigma}^2}{2} \sum_{k=0}^{K} \alpha_k^2 + \sigma^2 \sum_{k=0}^{K} \beta_k.
        \end{split}
    \end{equation*}
    Rearranging the above, by taking the limit as $K$ goes to infinity, the choice of the step size parameters $\alpha_k$ and $\beta_k$ and Assumption~\ref{asm.function}
    \begin{equation*}
        \begin{split}
            c_{\delta} \sum_{k=0}^{\infty} \beta_k^2 \E[|\lambda_k^-|^3] + \tfrac{1}{2} \sum_{k=0}^{\infty} \alpha_k\E[\|\nabla f(\hat{x}_k) \|_2^2] \leq f(x_0) - \bar{f} + \tfrac{L_g\hat{\sigma}^2}{2} \sum_{k=0}^{\infty} \alpha_k^2 + \sigma^2 \sum_{k=0}^{\infty} \beta_k < \infty,
        \end{split}
    \end{equation*}
    which implies 
    \begin{align} \label{fom.finitesum1}
        \sum_{k=0}^{\infty} \beta_k^2 \E[|\lambda_k^-|^3] < \infty \quad \text{and} \quad \sum_{k=0}^{\infty} \alpha_k \E[\|\nabla f(\hat{x}_k) \|_2^2] < \infty. 
    \end{align}
    Notice that by the descent step size choice \eqref{stepsize_eq}, the latter implies
    \begin{equation}\label{Lim:Result1}
        \liminf_{k \to \infty} \E[\|\nabla f(\hat{x}_k) \|_2] = 0.
    \end{equation}
    By Assumption~\ref{asm.function} and Condition \ref{fom.sto.hatgk}, 
    \begin{equation} \label{fom.normg_bound1}
        \E \left[\|\nabla f(x_{k+1}) - \nabla f(\hat{x}_k)\|_2 \right] \leq L_g \alpha_k (1+\hat{\theta}_k) \E[\|\nabla f(\hat{x}_k)\|_2] + L_g \alpha_k \hat{\sigma}.
    \end{equation}
    Combining the above with \eqref{Lim:Result1}
    \begin{align} \label{fom.normg_bound2}
            0 \leq \liminf_{k \to \infty} \E \left[\|\nabla f(x_{k+1})\|_2 \right] & \leq \liminf_{k \to \infty} \left( \E \left[\|\nabla f(\hat{x}_k)\|_2 \right] + \E \left[\|\nabla f(x_{k+1}) - \nabla f(\hat{x}_k)\|_2 \right] \right) \nonumber \\
            & \leq \liminf_{k \to \infty} \left((1+L_g \alpha_k(1+\hat{\theta}_k)) \E \left[\|\nabla f(\hat{x}_k)\|_2 \right] + L_g\alpha_k \hat{\sigma} \right) \nonumber \\
            & = \liminf_{k \to \infty}(1+L_g \alpha_k(1+\hat{\theta}_k)) \E \left[\|\nabla f(\hat{x}_k)\|_2 \right] + \lim_{k \to \infty} L_g\alpha_k \hat{\sigma} = 0
        \end{align}
    where the last equality is due to the fact that the sequence $\{L_g\alpha_k \hat{\sigma}\}$ is convergent. From this, the desired result follows.
\end{proof}
\bremark 
    By Lemma \ref{lmm: basineq} it follows that after each negative curvature step the prospective decrease (negative term in \eqref{important_eq} associated with negative curvature step) is $\mathcal{O}(\beta_k^2)$, whereas the error term (negative term in \eqref{important_eq} associated with negative curvature step) is $\mathcal{O}(\beta_k)$. To show a convergence result (to the solution), we need the summation of the error term to be finite which means $\beta_k$ has to be summable in this case. Consequently, $\beta_k^2$ is summable and we cannot derive convergence to a second-order stationary point. On the other hand, the descent steps provide decrease similar to the stochastic gradient methods (terms associated with $\alpha_k$ in \eqref{important_eq}), and as such in Theorem~\ref{thm.sto.conv.dimstep} only show convergence to first-order stationary points. 
\eremark

In this theorem, we present the convergence result for the setting where the variance is diminishing and we use constant step size.

\btheorem
\label{thm.sto.conv.dimvar}
    Suppose Assumption \ref{asm.function} and Conditions \ref{cond.sto.pk}, \ref{cond.sto.hk} and \ref{cond.sto.dk} hold with parameters $\gamma \in (0,1]$, $\gamma_H \in [0,\gamma)$, $\gamma_{\lambda} \in [0,1)$, $\gamma_H < \gamma (1-\gamma_{\lambda})$, $\delta \in (0, \infty)$, $\theta_k = \theta \in [0, 1)$, $\hat{\theta}_k = \hat{\theta} \in [0, \infty)$, and $\sigma_k> 0$, $\hat{\sigma}_k > 0$ where
    \begin{align*}
        \sum_{k=0}^{\infty} \sigma_k^2 < \infty, \quad \text{and} \quad \sum_{k=0}^{\infty} \hat{\sigma}_k^2 < \infty.
    \end{align*}
    Let the step size parameters $\alpha_k = \alpha$ and $\beta_k = \beta$ satisfy \eqref{eq.stepsizeconstant_sto}. Then, the iterates generated by Algorithm \ref{alg.sto.two-step} satisfy
    \begin{align*}
        \lim_{k \to \infty} \E \left[\| \nabla f(x_k)\|_2 \right] = 0 \quad \text{and} \quad \liminf_{k \to \infty} \E \left[ \lambda_{\min}(\nabla^2 f(x_k)) \right] \geq 0.
    \end{align*}
\etheorem
\bproof
    For reference, we restate the inequality from Lemma \ref{lmm: basineq} here
    \begin{align*}
        \E_k[f(x_{k+1})] \leq f(x_k) - \beta^2c_{\delta}  \E_{k}[|\lambda_{\min, k}^-|^3] - \tfrac{\alpha}{2} \E_{k} [\|\nabla f(\hat{x}_k) \|_2^2] + \tfrac{L_g}{2} \alpha^2 \hat{\sigma}_k^2 + \beta \sigma_k^2,
    \end{align*}
    where $c_{\delta}:= \tfrac{\delta^2}{4}  (\gamma - \tfrac{\gamma_H}{1-\gamma_{\lambda}}) (1-\gamma_{\lambda})^3$.
    Taking the total expectation, it follows that
    \begin{equation} \label{fom.sto.whole_exp}
        \E[f(x_{k+1})] \leq \E[f(x_k)] -  \beta^2 c_{\delta} \E[|\lambda_{\min, k}^-|^3] - \tfrac{\alpha}{2} \E [\|\nabla f(\hat{x}_k) \|_2^2] + \tfrac{L_g}{2} \alpha^2 \hat{\sigma}_k^2 + \beta \sigma_k^2.
    \end{equation}
    Summing from $k=0$ to $K$ and rearranging \eqref{fom.sto.whole_exp},
    \begin{equation*}
        \begin{split}
            \beta^2 c_{\delta}\sum_{k=0}^{K} \E[|\lambda_{\min, k}^-|^3] + \tfrac{\alpha}{2} \sum_{k=0}^{K} \E[\|\nabla f(\hat{x}_k) \|_2^2] & \leq f(x_0) - \E[f(x_{K+1})] + \tfrac{L_g\alpha^2}{2} \sum_{k=0}^{K} \hat{\sigma}_k^2 + \beta \sum_{k=0}^{K} \sigma_k^2.
        \end{split}
    \end{equation*}
    Taking the limit as $K \rightarrow \infty$, by Assumption~\ref{asm.function} and the conditions on $\sigma_k$ and $\hat{\sigma}_k$, 
    \begin{align*}
        \beta^2 c_{\delta} \sum_{k=0}^{\infty} \E[|\lambda_{\min, k}^-|^3] + \tfrac{\alpha}{2} \sum_{k=0}^{\infty} \E[\|\nabla f(\hat{x}_k) \|_2^2] 
        < \infty,
    \end{align*}
    from which it follows that
    \begin{align} \label{fom.finitesum2}
     \sum_{k=0}^{\infty} \E[|\lambda_{\min, k}^-|^3] < \infty \quad \text{and} \quad \sum_{k=0}^{\infty} \E[\|\nabla f(\hat{x}_k) \|_2^2] < \infty. 
    \end{align}
    The latter bound yields $\lim_{k \to \infty} \E[\|\nabla f(\hat{x}_k) \|_2] = 0$.  
    Using a similar argument as in the proof of Theorem~\ref{thm.sto.conv.dimstep}  (\eqref{fom.normg_bound1} and \eqref{fom.normg_bound2}), it follows that 
    \begin{align*}
        0 \leq \limsup_{k\to \infty} \E[\| \nabla f(x_{k+1})\|_2] \leq (1+L_g\alpha(1+\hat{\theta}_{k}))\lim_{k\to \infty} \E[\|\nabla f(\hat{x}_{k})\|_2] + L_g\alpha \lim_{k\to\infty} \hat{\sigma}_k = 0,
    \end{align*}
    which implies  $\lim_{k \to \infty} \E[\|\nabla f(x_k) \|_2] = 0$. The former bound implies 
    $\lim_{k \to \infty} \E[|\lambda_{\min, k}^-|^3] = 0$. Following the same arguments as in the deterministic setting (Theorem~\ref{thm.inexact.converge}), it follows that 
        $\liminf_{k \to \infty} \E[\lambda_{\min}(\nabla^2 f(x_k))] \geq 0$, 
    which concludes our proof.
\eproof

\bremark 
Theorem~\ref{thm.sto.conv.dimvar} shows that in the constant sufficiently small step size and diminishing variance setting, the iterates generated by Algorithm~\ref{alg.sto.two-step} converges to a second-order stationary point in expectation. This is a stronger result than that proven in Theorem~\ref{thm.sto.conv.dimstep} for the constant variance and diminishing step size setting. The reason for this is that the diminishing variance allows for more accurate gradient and Hessian estimations as the optimization progresses, and as such one can show \eqref{fom.finitesum2} with constant step sizes.
\eremark

We conclude this section with a corollary to Theorem~\ref{thm.sto.conv.dimvar} that provides an iterations complexity for Algorithm~\ref{alg.sto.two-step} in the diminishing variance and constant step size setting. 
\begin{corollary}
Consider any scalars $\epsilon_g$, $\epsilon_H>0$ and the conditions in Theorem~\ref{thm.sto.conv.dimvar}. With respect to Algorithm \ref{alg.sto.two-step}, 
the cardinality of the index set $\mathcal{G}(\epsilon_g) := \{k \in \mathbb{N}: \E[\|\nabla f({x}_{k}) \|_2] > \epsilon_g\}$ is at most $\mathcal{O}(\epsilon_g^{-2})$, and the cardinality of the index set $\mathcal{H}(\epsilon_H) := \{k \in \mathbb{N}: \E[|\lambda_{\min, k}^-|] > \epsilon_H \}$ is at most $\mathcal{O}(\epsilon_H^{-3})$. 
Hence, the number of iterations 
required until an iteration $k \in \mathbb{N}$ is reached with $\E[\|\nabla f(x_k) \|_2] \leq \epsilon_g$ and $\E[\lambda_{\min}(\nabla^2 f(x_k))] \geq -\epsilon_H$ is at most $\mathcal{O}(\max\{\epsilon_g^{-2}, \epsilon_H^{-3}\})$.
\end{corollary}
\bproof
    By Assumption \ref{asm.function}, the iteration update $x_{k+1} \leftarrow \hat{x}_k + \alpha d_k$ and Condition~\ref{fom.sto.hatgk},  
    \begin{align*}
        \epsilon_g^2 < 4\E[\|\nabla f(\hat{x}_k)\|_2^2] + 2\alpha L_g \hat{\sigma}_k^2.
    \end{align*}
    Thus, given $\epsilon_g \geq 0$, it follows that the set
    \begin{align*}
        \mathcal{G}(\epsilon_g) \subseteq \left\{k \in \N{}_{+}: \E[\|\nabla f(\hat{x}_{k-1})\|_2^2] + \tfrac{\alpha L_g}{2} \hat{\sigma}_{k-1}^2 > \tfrac{1}{4} \epsilon_g^2 \right\} := \hat{\mathcal{G}}(\epsilon_g).
    \end{align*}
    For all $k \in \N{}$, \eqref{fom.sto.whole_exp} holds, from which it follows that,
    \begin{align*}
            k \in \mathcal{G} (\epsilon_g) \subseteq \hat{\mathcal{G}}(\epsilon_g) \ & \Rightarrow \ \E[\|\nabla f(x_k)\|_2^2] > \epsilon_g^2 \ \Rightarrow \ \E[\|\nabla f(\hat{x}_{k-1})\|_2^2] + \tfrac{\alpha L_g}{2} \hat{\sigma}_{k-1}^2 > \tfrac{1}{4} \epsilon_g^2 \\
            & \Rightarrow \ \E[f(x_{k-1})] - \E[f(x_{k})] + \tfrac{3L_g \alpha^2}{4} \hat{\sigma}_{k-1}^2 + \beta \sigma_k^2 > \tfrac{\alpha}{8} \epsilon_g^2\\
            k \in \mathcal{H}(\epsilon_H) \ & \Rightarrow \ \E[|\lambda_{\min, k}^{-}|^3] > \epsilon_H^3 \\
            & \Rightarrow \ \E[f(x_{k-1})] - \E[f(x_{k})] + \tfrac{L_g \alpha^2}{2} \hat{\sigma}_{k-1}^2 + \beta \sigma_{k-1}^2 \geq  c_{\delta}\beta^2 \epsilon_H^3.
    \end{align*}
    By Assumption \ref{asm.function} and \eqref{fom.sto.whole_exp}, the sets $\mathcal{G}(\epsilon_g)$ and $\mathcal{H}(\epsilon_H)$ are both finite. By summing the reductions achieved in $f$ over the iterations, it follows that
\begin{equation*}
        \begin{split}
            & \infty > f(x_0)-\bar{f} + \tfrac{3L_g \alpha^2}{4} \sum_{k=0}^{\infty} \hat{\sigma}_k^2 + \beta \sum_{k=0}^{\infty} \sigma_k^2 \geq \tfrac{\alpha}{8} \epsilon_g^2 |\mathcal{G}(\epsilon_g)|, \\
            \text{and} \quad & \infty > f(x_0)-\bar{f} + \tfrac{L_g \alpha^2}{2} \sum_{k=0}^{\infty} \hat{\sigma}_k^2 + \beta \sum_{k=0}^{\infty} \sigma_k^2 \geq \tfrac{1}{4}C \beta^2\epsilon_H^3 |\mathcal{H}(\epsilon_H)|,
        \end{split}
    \end{equation*}
    which implies that $|\mathcal{G}(\epsilon_g)| \sim \mathcal{O}(\epsilon_g^{-2})$ and $|\mathcal{H}(\epsilon_H)| \sim \mathcal{O}(\epsilon_H^{-3})$ as expected.

    It follows that the number of iterations required until iteration $k \in \N{}$ is reached with $\E[\|\nabla f(x_k)\|_2] \leq \epsilon_g$ and $\E[|\lambda_{\min,k}^{-}|] \leq \epsilon_H$ is at most $\mathcal{O}(\max\{\epsilon_g^{-2}, \epsilon_H^{-3}\})$. Notice that $\lambda_{\min}(\nabla^2 f(x_k)) = \lambda_{\min,k}^{+} + \lambda_{\min,k}^{-}$ where $\lambda_{\min,k}^{+}:= \max\{\lambda_{\min}(\nabla^2 f(x_k)), 0\}$, so when $\E[|\lambda_{\min,k}^{-}|] \leq \epsilon_H$, it follows that
    \begin{align*}
        \E[\lambda_{\min}(\nabla^2 f(x_k))] = \E[\lambda_{\min,k}^{+}] + \E[\lambda_{\min,k}^{-}] \geq 0+\E[\lambda_{\min,k}^{-}]=-\E[|\lambda_{\min,k}^{-}|] \geq -\epsilon_H,
    \end{align*}
    which completes the proof.
\eproof
\bremark
With $\epsilon_g = \epsilon$ and $\epsilon_H = \sqrt{\epsilon}$ for some $\epsilon > 0$, the number of iterations to reach an $(\epsilon, \sqrt{\epsilon})$-second-order stationary point is at most $\mathcal{O}(\epsilon^{-2})$. 
Compared to the deterministic analogue of the corollary (Corollary~\ref{cor.complex1}), the complexity of the stochastic method matches that of the deterministic algorithm in terms of the tolerance $\epsilon$, however, the result is in expectation. 
\eremark

\section{Practical Algorithm} \label{sec.practical}

In this section, we describe a practical algorithm that exploits negative curvature. The algorithm does not take two steps at every iteration to update the iterate and instead exploits the power of CG to compute a step, does not explicitly compute the Hessian matrix and the associated minimum eigenvalue and instead relies solely on Hessian-vector products, utilizes only as much information as needed in the approximations (gradient and Hessian) employed and adjusts as optimization progresses, and dynamically selects the step size parameters. To this end, there are three main components: $(1)$ an adaptive sampling strategy for selecting the accuracy in the approximations employed at every iteration; $(2)$ the CG method with negative curvature detection and early stopping for computing the search direction; and, $(3)$ a practical step size selection scheme for setting the step size at every iteration. 

Given the current iterate $x_k \in \mathbb{R}^n$, sample sizes $b_k^g \in \N{}_{+}$, $b_k^H \in \N{}_{+}$, and sample sets $\mathcal{S}_k = \{\xi_{1}^g, \xi_{2}^g,\cdots \xi_{b_k^g}^g\}$ and $\mathcal{T}_k = \{\xi_{1}^H, \xi_{2}^H,\cdots \xi_{b_k^H}^H\}$ 
consisting of independent samples drawn at random from the distribution $\P$,  
the iterate is updated via
\begin{align}\label{eq.iterateupdate}
    x_{k+1} \gets x_k + \alpha_k d_k, \quad \text{where} \quad \nabla^2 f_{\mathcal{T}_k}(x_k) d_k = - \nabla f_{\mathcal{S}_k}(x_k),
\end{align}
$\alpha_k>0$ is the step size, and
\begin{align}\label{eq.sampled_gH}
    \nabla f_{\mathcal{S}_k}(x_k) = \tfrac{1}{|\mathcal{S}_k|} \sum_{\xi_i^g \in \mathcal{S}_k} \nabla F(x_k, \xi_i^g), \quad \text{and} \quad 
    \nabla^2 f_{\mathcal{T}_k}(x_k) = \tfrac{1}{|\mathcal{T}_k|} \sum_{\xi_i^H\in \mathcal{T}_k} \nabla^2 F(x_k, \xi_i^H).
\end{align}
The sample sets $\mathcal{S}_k$ and $\mathcal{T}_k$ are selected via an adaptive sampling strategy \cite{byrd2012sample}, the linear system in \eqref{eq.iterateupdate} is solved via the CG method \cite{hestenes1952methods} with negative curvature detection, and the step size $\alpha_k>0$ is set via an adaptive strategy \cite{bollapragada2018adaptive}. We discuss all the three components that are intimately connected in detail below. The step size strategy is well-defined and appropriate due to the adaptive sampling nature of the method, and the adaptive sampling strategy is well-suited with regard to CG and the approximations employed.

\subsection{Adaptive Sampling Strategy} \label{subsec.prac.adapsampling}

Adaptive sampling is a powerful technique that is used in stochastic optimization to control the accuracy of gradient (and possibly Hessian) estimates in a computationally efficient manner. Examples include the popular norm  \cite{carter1991global,byrd2012sample,berahas2022adaptive} and inner product  \cite{bollapragada2019exact,bollapragada2018progressive} tests. Inspired by the norm test and successful approximations thereof, and the conditions presented in Section~\ref{subsec.inexact.conditions} and \ref{subsec.stoch_setting.conds}, we customize said tests. Our algorithm requires gradient and Hessian approximations whose accuracy satisfy
\begin{align}
\label{fom.adapsampling.norm_test}
    \E_{k} \left[\| \nabla f_{\mathcal{S}_k}(x_k) - \nabla f(x_k) \|_2^2 \right] \leq \theta_k^2 \| \nabla f(x_k) \|_2^2 , \ \theta_k \in [0, 1), \\
    \E_{k}\left[\|(\nabla^2 f_{\mathcal{T}_k}(x_k) - \nabla^2 f(x_k)) d_k\|_2^2\right] \leq \theta_k^2 \E_{k}\left[\|d_k\|_2^2\right], \  \theta_k \in [0,1). \label{fom.adapsampling.norm_test_Hessian}
\end{align}
These conditions cannot be verified in practice without computing the true gradient and Hessian, and as such we approximate these conditions. Notice that we use information at the $k$th iteration to decide on the sample size at the next iteration. Given gradient $\mathcal{S}_k$ and Hessian $\mathcal{T}_k$ samples, our approach sets the samples sizes $b_{k+1}^g$ and $b_{k+1}^H$ as follows. For the gradient sample size, we approximate \eqref{fom.adapsampling.norm_test} via 
\begin{equation} \label{fom.adapsampling.approximation}
    \tfrac{\Var_{\xi_i^g \in \mathcal{S}_k}(\nabla F(x_k,\xi_i^g))}{|\mathcal{S}_k|}  \leq \theta_k^2 \| \nabla f_{\mathcal{S}_k}(x_k) \|_2^2,
\end{equation}
where the true variance is replaced by the sample variance, i.e., 
\begin{align}\label{fom.sample_var}
    \Var_{\xi_i^g \in \mathcal{S}_k}(\nabla F(x_k, \xi_i^g)) = \tfrac{1}{|\mathcal{S}_k|-1} \sum_{\xi_i^g \in \mathcal{S}_k} \|\nabla F(x_k,\xi_i^g) - \nabla f(x_k)\|_2^2.
\end{align} 
If \eqref{fom.adapsampling.approximation} is satisfied, then $b_{k+1}^g = |\mathcal{S}_{k}|$, otherwise, the new sample size is given by
\begin{equation} \label{fom.adapsampling.new_Sk}
    b_{k+1}^g = \left\lceil \tfrac{\Var_{\xi_i^g \in \mathcal{S}_k}(\nabla F(x_k,\xi_i^g))}{\theta_k^2 \| \nabla f_{\mathcal{S}_k}(x_k) \|_2^2} \right\rceil.
\end{equation}
The gradient approximation becomes increasingly accurate as the sample size $|\mathcal{S}_k|$ increases and has been shown to be efficient in practice \cite{bollapragada2018adaptive}. 
We apply the same idea for the Hessian sample size and condition \eqref{fom.adapsampling.norm_test_Hessian}. Here, the condition depends on the current search direction $d_k$. Namely, $b_{k+1}^H = |\mathcal{T}_{k}|$ if
\begin{equation}\label{fom.adapsampling.new_Tk2}
    \tfrac{\Var_{\xi_i^H \in \mathcal{T}_k}(\nabla^2 F(x_k,\xi_i^H) d_k)}{|\mathcal{T}_k|}  \leq \theta_k^2 \| d_k \|_2^2
\end{equation}
is satisfied, otherwise, 
\begin{equation} \label{fom.adapsampling.new_Tk}
    b_{k+1}^H = \left\lceil \tfrac{\Var_{\xi_i^H \in \mathcal{T}_k}(\nabla^2 F(x_k,\xi_i^H)d_k)}{\theta_k^2 \|d_k\|_2^2} \right\rceil.
\end{equation}

\subsection{Newton-CG with Negative Curvature Detection}

We utilize the CG method with negative curvature detection and early stopping to solve the linear system given in \eqref{eq.iterateupdate} and compute a search direction. We do so for three main reasons: (1) the approach exploits problem 
structure; (2) the approach can be implemented matrix-free; and (3) negative curvature detection and early stopping can be incorporated at no additional cost \cite{nocedal1999numerical}. At every iteration $k \in \N{}$, the CG method iteratively solves \eqref{eq.iterateupdate}. At every iteration of CG $j \in \N{}$, the approach either terminates with an approximate solution ($d_k$) to \eqref{eq.iterateupdate} defined as $\|-g_k-H_k d_k\|_2 \leq \epsilon_{CG} \|g_k\|_2$ for $\epsilon_{CG} > 0$ or a direction of sufficient negative curvature defined as $d_k^{\mathsf{T}} H_k d_k < -\epsilon_H \|d_k\|_2^2$ for $\epsilon_H>0$.

\subsection{Step Size Selection Scheme}

Our algorithm makes use of an adaptive step size selection scheme. Such approaches are fragile in fully stochastic regimes, and even in adaptive sampling settings, where the accuracy in the approximations employed can be controlled, care needs to be taken to ensure efficiency and robustness. To this end, we employ a sufficient decrease backtracking line search mechanism that is cautious in the choice of the initial trial step size \cite[Section 2.2]{bollapragada2018progressive}; see 
Algorithm~\ref{alg.prac.backtracking}. Specifically, we employ a variance-based initial trial step size, 
\begin{align}\label{eq.stepsize_var}
    \alpha_k \leftarrow \left(1+ \tfrac{\Var_{\xi_i^g \in \mathcal{S}_k}(\nabla F(x_k, \xi_i^g))}{|S_k| \|\nabla f_{S_k}(x_k)\|_2^2} \right)^{-1},
\end{align}
where the variance estimate is given in \eqref{fom.sample_var}. In the deterministic setting, \eqref{eq.stepsize_var} reduces to unity.
\begin{algorithm}
    \caption{Backtracking line search}\label{alg.prac.backtracking}
    \begin{algorithmic}[1]
        \Require{$x_k \in \Real^n$, $g_k \in \Real^n$, $d_k \in \Real^n$, 
        $S_k = \{\xi_{1}^g, \xi_{2}^g,\cdots \xi_{b_k^g}^g\}$, $c_1 \in (0,1)$, $\eta \in (0,1)$}
            \State set the initial step size $\alpha_k$ via \eqref{eq.stepsize_var}
            \State compute trial function value $f_{\text{trial}} \leftarrow f_{S_k}(x_k + \alpha_k d_k)$
            \While{$f_{\text{trial}} > f_{S_k}(x_k) + c_1 \alpha_k  g_k^{\mathsf{T}}  d_k$}
                \State reduce step size $\alpha_k \leftarrow \eta \alpha_k$
                \State compute trial function value $f_{\text{trial }} \leftarrow f_{S_k}(x_k + \alpha_k d_k)$
            \EndWhile
    \end{algorithmic}
\end{algorithm}

\subsection{Algorithm}
We now present our practical algorithm which contains the three components mentioned above (Algorithm \ref{alg.prac.ncas}). The main (while) loop consists of classical CG iterations. The search direction is computed iteratively and is either an approximate solution of \eqref{eq.iterateupdate} or a direction of sufficient negative curvature. The step size is computed via an adaptive procedure. Finally, using current information, the sample sizes are set for the next iteration. 

\begin{algorithm}
    \caption{A Newton-CG Method with Negative Curvature Detection and Adaptive Sample Size Selection (\texttt{NCAS})} \label{alg.prac.ncas}
    \begin{algorithmic}[1]
        \Require $x_0 \in \Real^n$, 
        $b_0^g \in \N{}_+$, $b_0^H \in \N{}_+$, $\epsilon_{CG} > 0$, $\epsilon_H > 0$, $\theta \in [0, 1)$, 
        $N_{CG} \in \N{}_{+}$
        \ForAll{$k \in \{0, 1, \dots\}$}
        \State choose the sample sets $\mathcal{S}_k = \{\xi_1^g,\xi_2^g,\cdots, \xi_{b_k^g}^g\}$ and $\mathcal{T}_k = \{\xi_{1}^H, \xi_{2}^H,\cdots \xi_{b_k^H}^H\}$
        \State compute $g_k = \nabla f_{\mathcal{S}_k}(x_k)$ and $H_k = \nabla^2 f_{\mathcal{T}_k}(x_k)$
        \State set $\bar{H}_k = H_k+2\epsilon_H I$ \label{alg.prac.line4}
        \State set $z_0 \leftarrow 0$, $r_0 \leftarrow g_k$ , $p_0 \leftarrow -r_0$, $d_k \leftarrow p_0$, $j \leftarrow 0$  \label{alg.prac.line5}
        \If{$p_0^{\mathsf{T}} H_k p_0 < -\epsilon_H \|p_0\|_2^2$} \label{alg.prac.line6}
            \State set $d_k \leftarrow p_0$ and \textbf{go to Line \ref{alg.prac.line23}} \label{alg.prac.line7}
        \EndIf
        \While{$j \leq N_{CG}$}
            \State $s_j \leftarrow r_j^{\mathsf{T}} r_j / p_j^{\mathsf{T}}\bar{H}_k p_j$, $z_{j+1} \leftarrow z_j + s_j p_j$, $r_{j+1} \leftarrow r_j + s_j \bar{H}_k p_j$ \label{alg.prac.line9}
            \State $t_{j+1} \leftarrow r_{j+1}^{\mathsf{T}} r_{j+1}/r_j^{\mathsf{T}} r_j$, $p_{j+1} \leftarrow -r_{j+1} + t_{j+1} p_j$, $d_k \leftarrow z_{j+1}$, $j \leftarrow j+1$ \label{alg.prac.line10}
            \If{$\|r_{j}\| \leq \epsilon_{CG} \|r_{0}\|$}
                \State \textbf{go to Line \ref{alg.prac.line23}}  \label{alg.prac.line12}
            \ElsIf{$p_{j}^{\mathsf{T}} H_k p_{j} < -\epsilon_H \|p_{j}\|_2^2$} \label{alg.prac.line13}
                \If{$p_{j}^{\mathsf{T}}g_k \leq 0$}
                    \State set $d_k \leftarrow p_{j}$ and \textbf{go to Line \ref{alg.prac.line23}}
                \Else
                    \State set $d_k \leftarrow -p_{j}$ and \textbf{go to Line \ref{alg.prac.line23}}
                \EndIf
            \ElsIf{$z_{j}^{\mathsf{T}} H_k z_{j} < -\epsilon_H \|z_{j}\|_2^2 $}
                \If{$z_{j}^{\mathsf{T}}g_k \leq 0$}
                    \State set $d_k \leftarrow z_{j}$ and \textbf{go to Line \ref{alg.prac.line23}}
                \Else
                    \State set $d_k \leftarrow -z_{j}$ and \textbf{go to Line \ref{alg.prac.line23}} \label{alg.prac.line22}
                \EndIf
            \EndIf
        \EndWhile
        \State compute $\alpha_k$ via \textbf{Algorithm \ref{alg.prac.backtracking}} \label{alg.prac.line23}
        \State set $x_{k+1} \leftarrow x_k + \alpha_k d_k$ 
        \State set the sample sizes $b_{k+1}^g$ and $b_{k+1}^H$ via \eqref{fom.adapsampling.approximation}-\eqref{fom.adapsampling.new_Sk} and \eqref{fom.adapsampling.new_Tk2}-\eqref{fom.adapsampling.new_Tk}, respectively \label{alg.prac.line25}
        \EndFor
    \end{algorithmic}
\end{algorithm}

\bremark We make a few remarks about Algorithm \ref{alg.prac.ncas}. 
\begin{itemize}
    \item \textbf{Hessian correction:} We employ a corrected Hessian (Line 4) in the CG method (Lines 9-12) where we added $2\epsilon_H I$ to the Hessian matrix $H_k$. This correction ensures that when $\lambda_{\min}(H_k) \geq -\epsilon_H$, the modified Hessian is sufficiently positively-definite, i.e.,  $\lambda_{\min} (\bar{H}_k) \geq \epsilon_H$, thereby avoiding stability issues in the computation of quantities such as $s_j$  involving $p_j^T\bar{H}_kp_j$ used in the CG method (Line 9). We note that the corrected Hessian is only used for the CG computations, and negative curvature directions are evaluated using the unmodified Hessian approximation. Specifically, when $\lambda_{\min}(H_k)$ is positive or sufficiently negative, the algorithm will get Newton's direction or a (sufficiently) negative curvature direction even without the correction. However, when $\lambda_{\min}(H_k) \in [-2\epsilon_H, 0]$, it might be hard to detect such direction and to implement CG process since $p_j^{\mathsf{T}}H_kp_j$ in line \ref{alg.prac.line9} could be non-positive. The use of $\bar{H}_k$ in the CG process can help with this case. Specifically, if $p_j$ is not a sufficiently negative curvature direction, then by line \ref{alg.prac.line6} and \ref{alg.prac.line13}, we have $p_j^{\mathsf{T}}\bar{H}_kp_j \geq \epsilon_H \|p_j\|_2^2 > 0$ which is positive enough to proceed the CG process. Such a correction approach has been employed successfully in the literature~\cite{royer2018complexity}.

    \item \textbf{CG method:} We use the CG method with negative curvature detection to either compute an approximate solution to $\bar{H}_k d = -g_k$ or compute a direction of negative curvature with respect to the matrix $H_k$ (Lines \ref{alg.prac.line5}-\ref{alg.prac.line22}).  
   Line \ref{alg.prac.line5} is the initialization process and Lines \ref{alg.prac.line9}-\ref{alg.prac.line12} are the steps of the standard CG procedure. The CG method has two parameters, $N_{CG}$ the maximum number of CG iterations and $\epsilon_{CG}$ the accuracy of CG. If the maximum number of iterations is reached, the algorithm returns the current estimate $z_{j+1}$ (Line~\ref{alg.prac.line10}).

    \item \textbf{Negative curvature detection:} Directions of negative curvature, if present, are detected on Lines \ref{alg.prac.line6}-\ref{alg.prac.line7} and Lines \ref{alg.prac.line13}-\ref{alg.prac.line22} of Algorithm~\ref{alg.prac.ncas}, and are an add-on to the CG routine. As a result, these negative curvature checks 
    come at negligible additional cost. Since any vector generated by the CG process could be a direction of negative curvature, 
    both directions $p_j$ and $z_j$ are tested.

    \item \textbf{Step Size:} The step size is selected (Line~\ref{alg.prac.line23}, Algorithm~\ref{alg.prac.ncas}) via Algorithm \ref{alg.prac.backtracking}. This approach is well-defined and is guaranteed to terminate finitely with an appropriate step size that is bounded away from zero \cite{bollapragada2018progressive}.
    
    \item \textbf{Sample size selection:} The sample sizes of gradient and Hessian are updated on Line \ref{alg.prac.line25}. For efficiency, and to avoid multiple gradient computations, the sample sizes for the next iteration are computed using the information at the current iterate. The updating rules follow those described in Section \ref{subsec.prac.adapsampling} and are designed to achieve a balance between optimality and complexity \cite{byrd2011use, byrd2012sample, martens2010deep}. 
    To avoid large batch size increases due only to the stochasticity, the rate of increase is capped by a 
    parameter $\zeta > 1$, i.e., $b_{k}^g \leq b_{k+1}^g \leq \lceil\zeta b_{k}^g\rceil$ and $b_{k}^H \leq b_{k+1}^H \leq \lceil\zeta b_{k}^H \rceil$. 
\end{itemize}
\eremark

\section{Numerical Experiments} \label{sec.num_exp}

In this section, we demonstrate the empirical performance of Algorithm \ref{alg.prac.ncas} on two nonconvex machine learning tasks, Robust Regression and Tukey Biweight, on datasets from the LIBSVM collection \cite{chang2011libsvm}. 
Our numerical study has two components. We first investigate the sensitivity and robustness of our method to the main algorithmic parameters (Section~\ref{subsec.exp.sensitivity}). We then compare the empirical performance of our method 
to that of a stochastic gradient method with adaptive sampling, a Trust Region Newton-CG method with adaptive sampling, and Algorithm \ref{alg.prac.ncas} without adaptive sampling (Section~\ref{subsec.exp.performance}). The goals of this comparison are to illustrate the power of negative curvature directions combined with the CG method and 
the power of the 
adaptive sampling strategy. All experiments were conducted in Matlab R2021b.

\subsection{Problem Specification, Algorithms, and Evaluation Metrics} \label{subsec.exp.settings}

We consider two machine learning tasks (Robust Regression \cite{chan2022nonlinear} and Tukey Biweight \cite{loh2017statistical, yu2017robust}). 
Both problems are nonconvex and data-dependent. Let $m$ denote the number of data points, $n$ denote the number of features, and $a_i \in \R{n}$ and $b_i \in \{-1,1\}$ denote the feature vector and the associated label, respectively, for $i\in\{1, \dots, m\}$. 
We consider three datasets from the LIBSVM collection \cite{chang2011libsvm} (\texttt{australian}: $n=14$, $m=621$; \texttt{mushroom}: $n=112$, $m=5500$; \texttt{splice}: $n=60$, $m=3175$). 
The robust regression problem is formally defined as follows,
    \begin{align*}
        \min_{x\in\R{n}} f_{RR}(x) = \tfrac{1}{m} \sum_{i=1}^{m} \phi(a_i^{\mathsf{T}}x - b_i), \quad \text{where} \quad \phi(t) = \tfrac{t^2}{1+t^2}.
    \end{align*}
    The Tukey Biweight problem is formally defined as follows, 
    \begin{align*}
        \min_{x\in\R{n}} f_{TB}(x) = \tfrac{1}{m} \sum_{i=1}^{m} \rho_{\sqrt{6}}(a_i^{\mathsf{T}}x - b_i), \quad \text{where} \quad 
        \rho_{\sqrt{6}}(t) = \begin{cases}
            \tfrac{t^6}{216} - \tfrac{t^4}{12} + \tfrac{t^2}{2} & \text{if } |t| \leq \sqrt{6}, \\
            1
            & \text{otherwise}.
        \end{cases} 
    \end{align*}

In order to investigate the merits and limitations of \texttt{NCAS} (Algorithm \ref{alg.prac.ncas}) and the three key components, we compare against the \texttt{SGAS}, \texttt{TRAS}, and \texttt{NC} methods. 
\begin{itemize}
    \item \texttt{SGAS}: We compare against the stochastic gradient method with adaptive sampling \cite{byrd2012sample}. 
    The only difference between \texttt{SGAS} and Algorithm \ref{alg.prac.ncas} is that \texttt{SGAS} does not utilize negative curvature (or the CG procedure to compute a search direction) and the search direction is set as the negative gradient. \texttt{SGAS} is a special case of Algorithm \ref{alg.prac.ncas} where $N_{CG} = 0$. By comparing these two methods, we want to showcase the power of negative curvature directions. 
    \item \texttt{NC}: We compare against a variant of Algorithm \ref{alg.prac.ncas} where the sample size is fixed to the full batch for both the gradient and Hessian approximations \cite{royer2020newton}. By comparing these two methods, we want to showcase the efficiency of the adaptive sampling strategy. 
     \item \texttt{TRAS}: We compare against a trust-region variant of Algorithm~\ref{alg.prac.ncas}. Specifically, \texttt{TRAS} utilizes the Steihaug version of CG \cite{steihaug1983conjugate} to compute a step within a trust region. The method utilizes the same sample size selection scheme as Algorithm~\ref{alg.prac.ncas} (discussed in Section~\ref{subsec.prac.adapsampling}). By comparing these two methods, we want to investigate the merits and limitations of the line search approach in the inexact setting.
\end{itemize}

The \texttt{SGAS}, \texttt{NC} and \texttt{NCAS} algorithms  make use of Algorithm~\ref{alg.prac.backtracking} to adaptively set the step size. The methods are summarized in Table~\ref{table.experiment}.
\begin{table}[h]
    \renewcommand\arraystretch{1.2}
    \centering\resizebox{\columnwidth}{!}{
    \begin{tabular}{c|c|c|c}
        Algorithm & 
            Negative Curvature
            Detection 
        & Adaptive Sampling & 
            Line Search (\texttt{LS})\slash 
            Trust Region (\texttt{TR})
       \\
        \hline\hline
        
            \yellow{\texttt{SGAS}} \cite{byrd2012sample}
        &
        & \Checkmark & \texttt{LS}
        \\
        \hline
            \green{\texttt{NC}} \cite{royer2020newton}
        & \Checkmark &  & \texttt{LS}
        \\
        \hline
            \cyan{\texttt{TRAS}} \cite{curtis2021trust, ha2023common}
        & \Checkmark & \Checkmark & \texttt{TR} \\
        \hline
        \red{\texttt{NCAS}} [Algorithm~\ref{alg.prac.ncas}] &  \Checkmark & \Checkmark & \texttt{LS}
        \\
    \end{tabular}
    }
    \caption{Algorithms compared in the following experiments.} 
    \label{table.experiment}
\end{table}

In terms of the algorithmic parameters, for algorithms that use Algorithm \ref{alg.prac.backtracking} to compute the step size, $c_1 = 10^{-4}$, $\eta = 0.5$. For \texttt{TRAS}, the trust-region related parameters are set to standard values $c_1 = 0.25$, $c_2 = 0.75$, and 
$\delta_0 = 1$ \cite{nocedal1999numerical}. For 
\texttt{NC} and \texttt{NCAS}, the Hessian accuracy parameter is set to  $\epsilon_H = 10^{-3}$. Algorithms \texttt{NCAS}, \texttt{SGAS}, and \texttt{TRAS} use the adaptive sampling strategy to decide the sample size for each iteration. The initial sample sizes are set to  $S_{\text{grad}} = 2$ and $S_{\text{Hess}} = 2$, the accuracy parameter $\theta = 0.9$, and the maximum increase rate $\zeta = 2$. The accuracy parameter of the CG process is $\epsilon_{CG} = 10^{-6}$ and the maximum number of CG iterations is $N_{CG} = 10$. 

In all experiments, we compare the methods in terms of the norm of gradient, the minimum eigenvalue, the sample size for gradient and Hessian, and step size/trust region radius. We present the evolution of these measures with respect to the iterations and  \textit{total evaluations}, which takes into consideration function, gradient, and Hessian evaluations, i.e.,  
$
    \textit{Total Evaluations} = 1 \cdot \textit{Evaluations}({f}) + 2 \cdot\textit{Evaluations}({\nabla f}) + 4 \cdot\textit{Evaluations}({\nabla^2 f}*s)
$ \cite{nocedal1999numerical}. 
All algorithms were terminated on a budget of total evaluations.

\subsection{Sensitivity Analysis} \label{subsec.exp.sensitivity}

In this subsection, we investigate the robustness of \texttt{NCAS} to five user-defined parameters. Three of these parameters are related to the CG subroutine: $(1)$ the eigenvalue accuracy parameter $\epsilon_H \in \{10^{-1}, 10^{-2}, 10^{-3}, 10^{-4} \}$ (Figure~\ref{fig:first}); $(2)$ the CG accuracy parameter $\epsilon_{CG} \in \{10^{-2}, 10^{-4}, 10^{-6}, 10^{-8}\}$ (Figure~\ref{fig:second}); and, $(3)$ the maximum CG iterations $N_{CG} \in \{0, 1, 5, 10, 100\}$ (Figure~\ref{fig:fourth}). The other two parameters are related to the adaptive sampling strategy: $(1)$ the adaptive sampling accuracy parameter $\theta \in \{0.1, 0.5, 0.9, 0.999\}$ (Figure~\ref{fig:third}); and, $(2)$ the maximum increase rate $\zeta \in \{1.1, 1.5, 2, 5, 10\}$ (Figure~\ref{fig:fifth}). We present results on the robust regression problem and the \texttt{australian} dataset. The default parameters in this experiment are $\epsilon_H = 10^{-3}$, $\epsilon_{CG} = 10^{-6}$, $\theta = 0.9$, $N_{CG} = 10$, and 
$\zeta = 2$. Algorithm \ref{alg.prac.ncas} with these default parameters serves as a baseline, and is shown in red in Figures~\ref{fig.sensitivity1} and \ref{fig.sensitivity2}.

The results suggest that the default parameters are often competitive. With respect to the parameters associated with the CG method, Algorithm~\ref{alg.prac.ncas} is robust across of wide range of parameter values. That said, there are certain parameter settings for which Algorithm~\ref{alg.prac.ncas} show significant slow-down, e.g., $\epsilon_{H} = 10^{-1}$ or $N_{CG} = 0, 1$. The reason for this is that in the former setting, the Hessian approximations are perturbed too much and useful second-order information is lost, and in the latter setting, not enough (if any) CG iterations are performed and the search directions are essentially gradient directions. With respect to the parameters associated with the adaptive sampling strategy, the performance is robust except when the maximum sample size increase is small, i.e., $\zeta = 1.1$.  The convergence rate in this case is slower as the sample size increase is restricted too much and the approximations employed are not accurate enough. As expected, the sample size increases faster when $\theta$ is small and $\zeta$ is large. Similar behavior was observed on other datasets. 

\begin{figure}[]
\centering
\begin{subfigure}{1\textwidth}
    \includegraphics[width=0.235\textwidth]{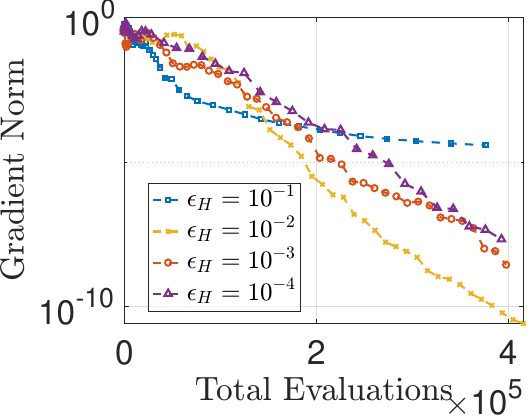}
    \includegraphics[width=0.24\textwidth]{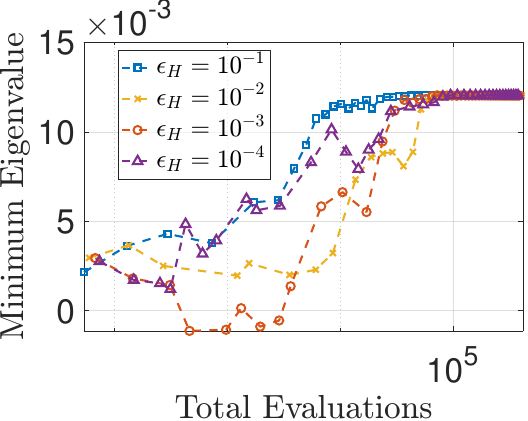}  \includegraphics[width=0.23\textwidth]{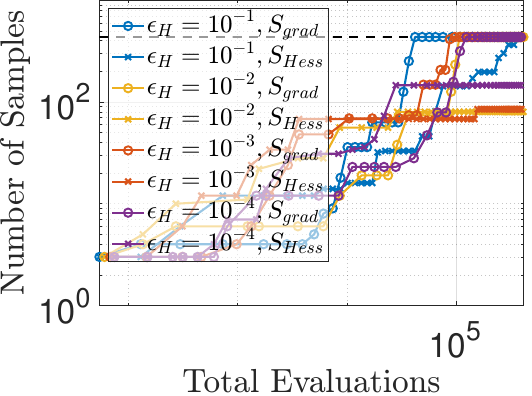} \includegraphics[width=0.23\textwidth]{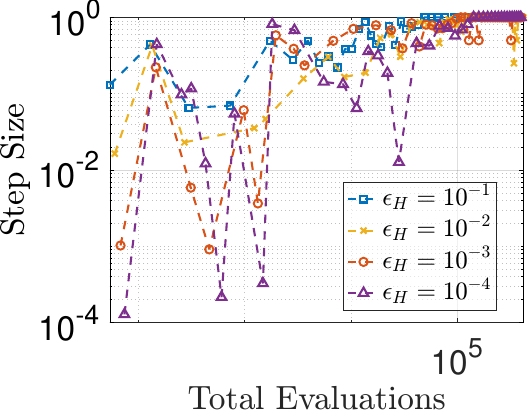}
    \caption{Sensitivity with respect to the eigenvalue accuracy parameter, $\epsilon_H \in \{10^{-1}, 10^{-2}, 10^{-3}, 10^{-4} \}$.}
    \label{fig:first}
\end{subfigure}
\hfill
\begin{subfigure}{1\textwidth}
    \includegraphics[width=0.235\textwidth]{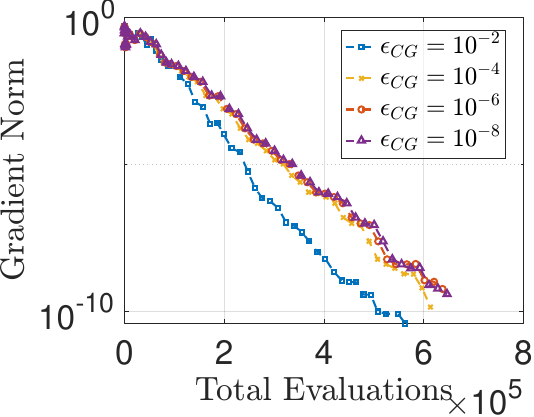} \includegraphics[width=0.24\textwidth]{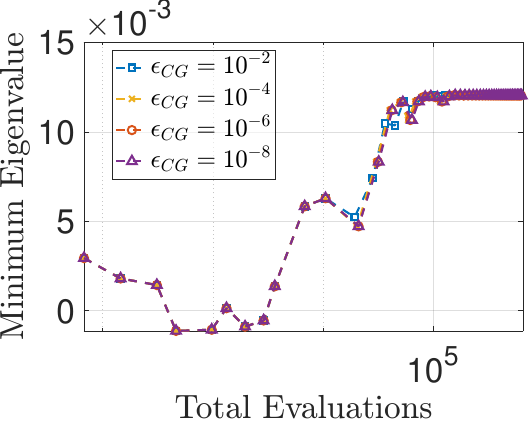}  \includegraphics[width=0.23\textwidth]{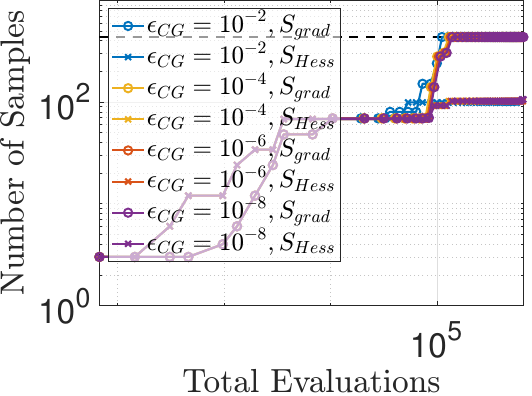}  \includegraphics[width=0.23\textwidth]{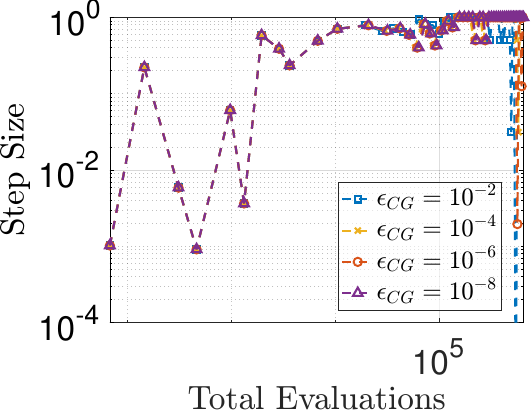}
    \caption{Sensitivity with respect to the CG accuracy parameter $\epsilon_{CG} \in \{10^{-2}, 10^{-4}, 10^{-6}, 10^{-8}\}$.}
    \label{fig:second}
\end{subfigure}
\hfill
\begin{subfigure}{1\textwidth}
    \includegraphics[width=0.24\textwidth]{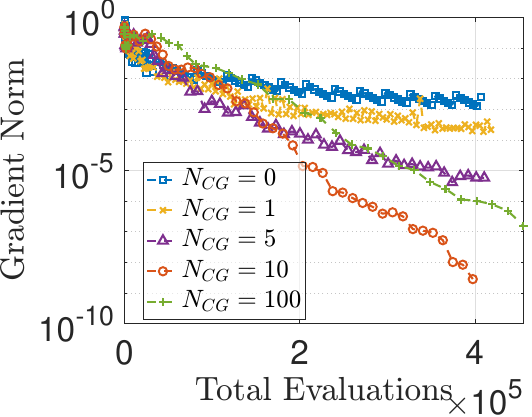}
     \includegraphics[width=0.24\textwidth]{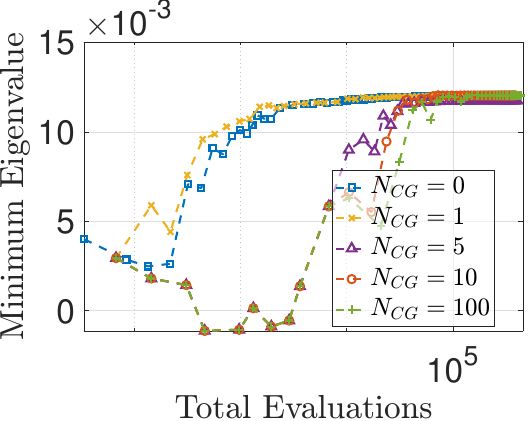} \includegraphics[width=0.235\textwidth]{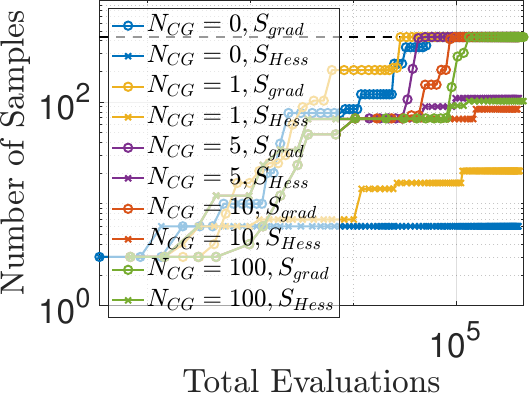}  \includegraphics[width=0.23\textwidth]{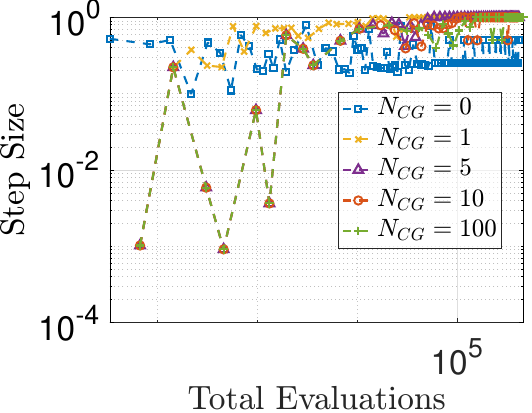}
    \caption{Sensitivity with respect to the maximum CG iterations $N_{CG} \in \{0, 1, 5, 10, 100\}$.}
    \label{fig:fourth}
\end{subfigure}
\caption{Sensitivity analysis of \texttt{NCAS} (Algorithm~\ref{alg.prac.ncas}) on robust regression problem (\texttt{australian} dataset) with respect to the parameters associated with the CG subroutine ($\epsilon_H$, $\epsilon_{CG}$, and $N_{CG}$) in terms of total evaluations.}
\label{fig.sensitivity1}
\end{figure}

\begin{figure}[]
\centering
\begin{subfigure}{1\textwidth}
    \includegraphics[width=0.235\textwidth]{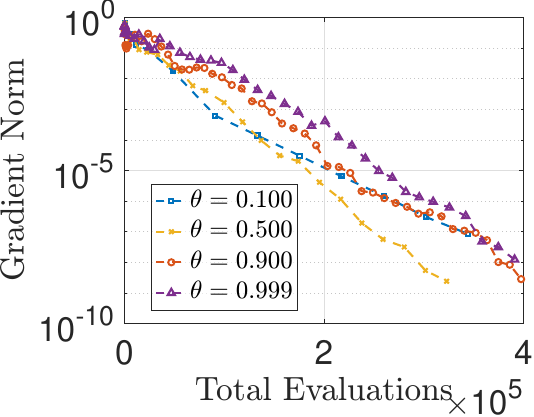} \includegraphics[width=0.24\textwidth]{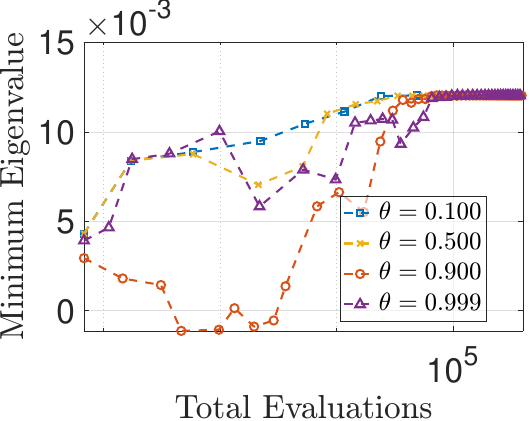}  \includegraphics[width=0.23\textwidth]{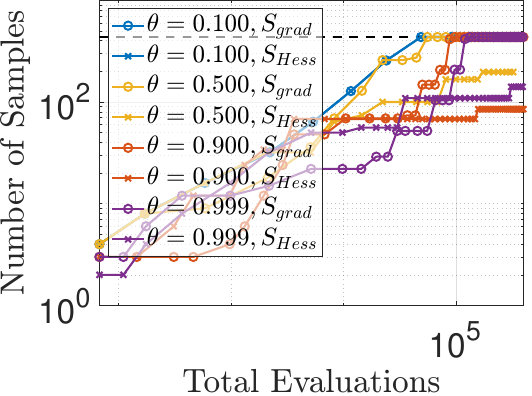}  \includegraphics[width=0.23\textwidth]{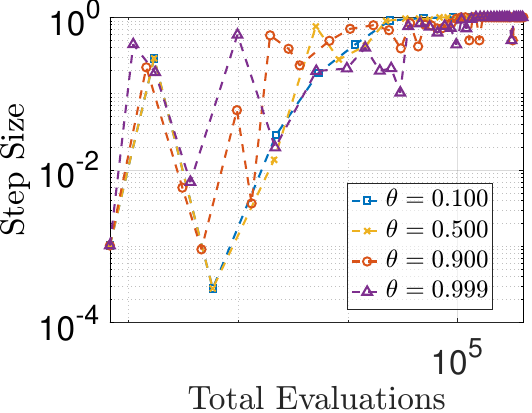}
    \caption{Sensitivity with respect to the adaptive sampling accuracy parameter $\theta \in \{0.1, 0.5, 0.9, 0.999\}$.}
    \label{fig:third}
\end{subfigure}
\hfill
\begin{subfigure}{1\textwidth}
    \includegraphics[width=0.23\textwidth]{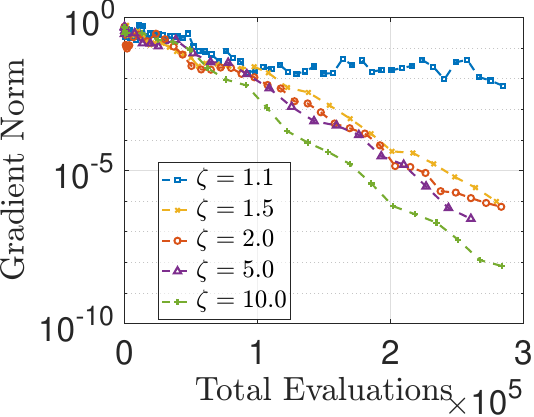} \includegraphics[width=0.234\textwidth]{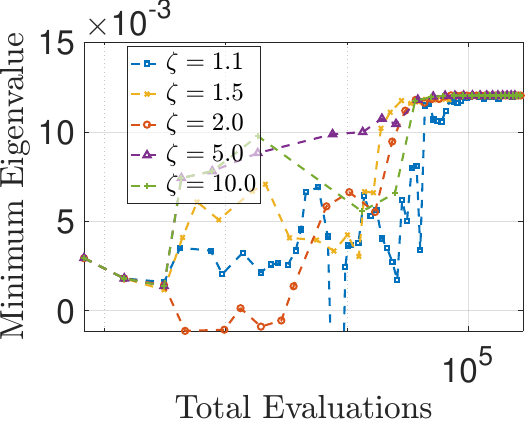}  \includegraphics[width=0.226\textwidth]{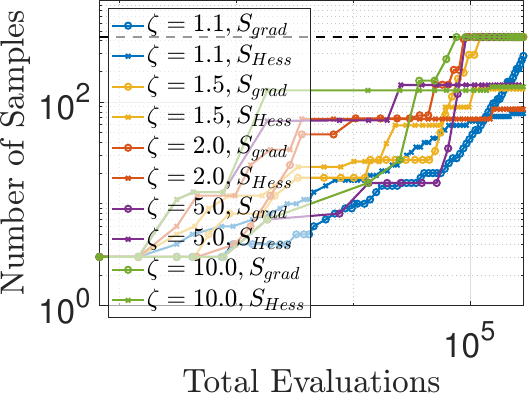}  \includegraphics[width=0.23\textwidth]{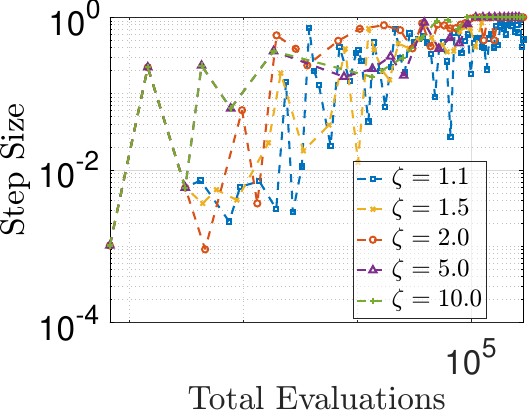}
    \caption{Sensitivity with respect to the adaptive sampling increase rate $\zeta \in \{1.1, 1.5, 2, 5, 10\}$.}
    \label{fig:fifth}
\end{subfigure}
\caption{Sensitivity analysis of \texttt{NCAS} (Algorithm~\ref{alg.prac.ncas}) on robust regression problem (\texttt{australian} dataset) with respect to the parameters associated with the adaptive sampling scheme ($\theta$, and $\zeta$) in terms of total evaluations.}
\label{fig.sensitivity2}
\end{figure}

\newpage

\subsection{Comparative Analysis} \label{subsec.exp.performance}

In this subsection, we compare the performance of \texttt{NCAS} with the methods described in Section \ref{subsec.exp.settings} on the robust regression (Figures~\ref{fig.rr.aus} and \ref{fig.rr.mush}) and the Tukey Biweight (Figures~\ref{fig.tb.aus} and \ref{fig.tb.mush}) problems on three datasets (\texttt{australian}, \texttt{mushroom}, \texttt{splice}). For all methods, problems, and datasets, we present the evolution of the gradient norm, the minimum eigenvalue, the sample size, and the step size/trust region with respect to iteration and total evaluations. We note that the \texttt{NC} method does not appear on any of  the sample size plots since it uses all samples at every iteration.

Figure~\ref{fig.rr.aus} shows the results on the robust regression problem with the \texttt{australian} dataset. In terms of iterations, the plots illustrate the quality of the search directions computed. As expected, the quality of the steps of the \texttt{NC} method is superior than the other methods due to the fact that the method uses exact function, gradient, and Hessian information to compute a step. In terms of total evaluations, our proposed method \texttt{NCAS} appears to be competitive as it makes frugal uses of the samples used at every iteration. Another interesting observation is that neither second-order adaptive sampling methods (\texttt{NCAS} and \texttt{TRAS}) reached the full Hessian sample size within the given budget. Similar behavior was observed on the \texttt{mushroom} and \texttt{splice} data sets (Figure~\ref{fig.rr.mush}). In fact, the advantages of adaptive sampling approaches and specifically \texttt{NCAS} are more pronounced in these larger (in terms of total samples) data sets. Finally, across the three data sets, the benefits of utilizing directions of negative curvature are clear.

\begin{figure}[]
    \renewcommand\arraystretch{1.2}
    \centering\resizebox{\columnwidth}{!}{
    \begin{subfigure}{0.3\textwidth}
        \centering
        \includegraphics[width=\textwidth]{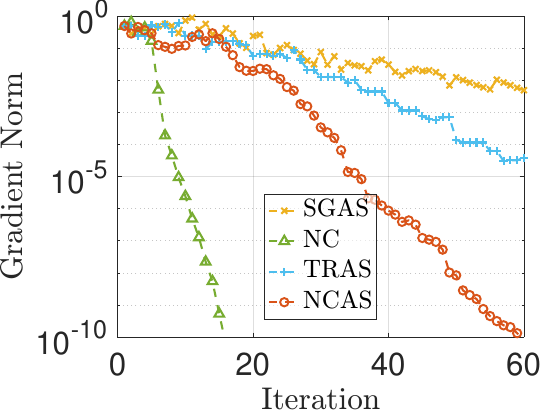}
    \end{subfigure}
    ~
    \begin{subfigure}{0.3\textwidth}
        \centering
        \includegraphics[width=\textwidth]{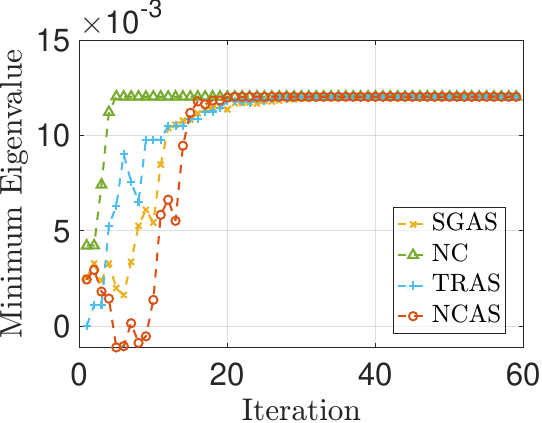}
    \end{subfigure}
    ~
    \begin{subfigure}{0.31\textwidth}
        \centering
        \includegraphics[width=0.93\textwidth]{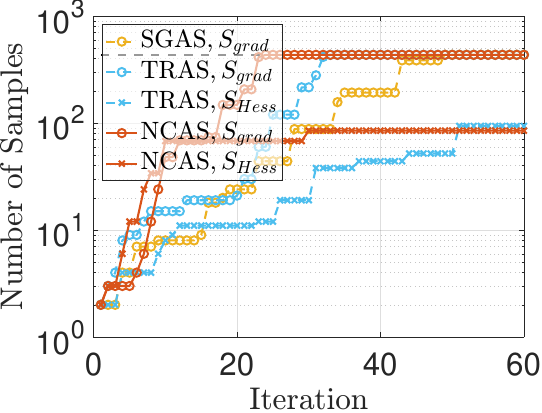}
    \end{subfigure}
    ~
    \begin{subfigure}{0.31\textwidth}
        \centering
        \includegraphics[width=\textwidth]{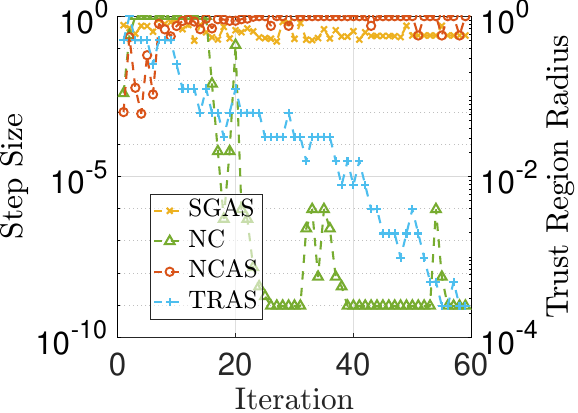}
    \end{subfigure}}\\

    \renewcommand\arraystretch{1.2}
    \centering\resizebox{\columnwidth}{!}{
    \begin{subfigure}{0.3\textwidth}
        \centering
        \includegraphics[width=\textwidth]{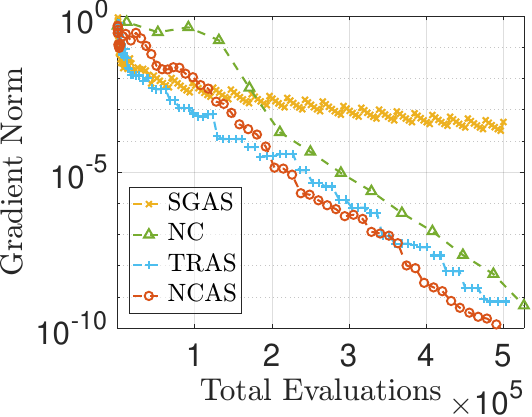}
    \end{subfigure}
    ~
    \begin{subfigure}{0.3\textwidth}
        \centering
        \includegraphics[width=\textwidth]{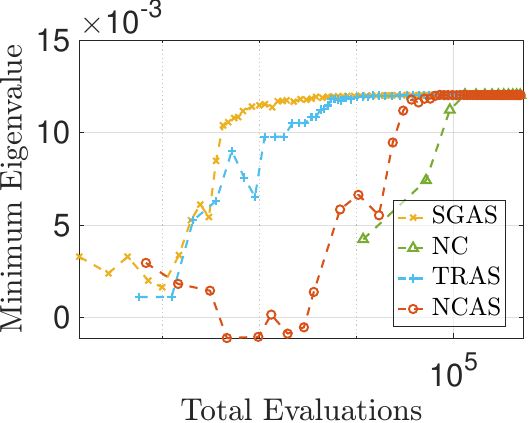}
    \end{subfigure}
    ~
    \begin{subfigure}{0.305\textwidth}
        \centering
        \includegraphics[width=\textwidth]{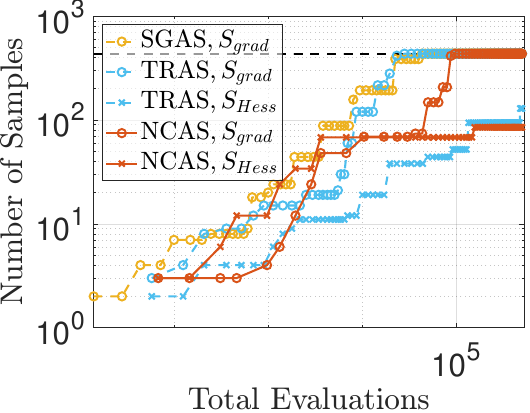}
    \end{subfigure}
    ~
    \begin{subfigure}{0.32\textwidth}
        \centering
        \includegraphics[width=\textwidth]{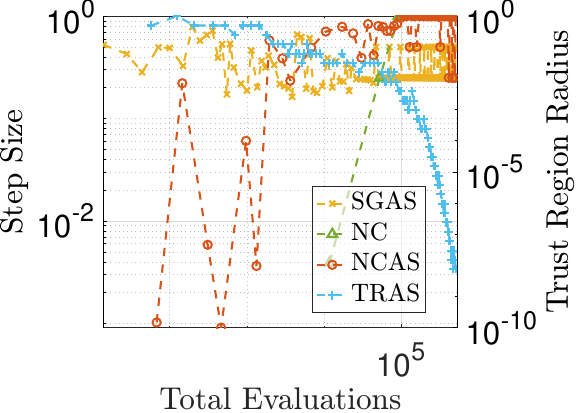}
    \end{subfigure}}
    \caption{Performance of \texttt{SGAS}, \texttt{NC}, \texttt{TRAS}, and \texttt{NCAS} on robust regression problem (\texttt{australian} dataset) in terms of gradient norm, minimum eigenvalue, sample size, and step size/trust region radius. First row iterations; Second row total evaluations.    
    }  \label{fig.rr.aus}
\end{figure}

\begin{figure}[]
    \renewcommand\arraystretch{1.2}
    \centering\resizebox{\columnwidth}{!}{
    \begin{subfigure}{0.3\textwidth}
        \centering
        \includegraphics[width=\textwidth]{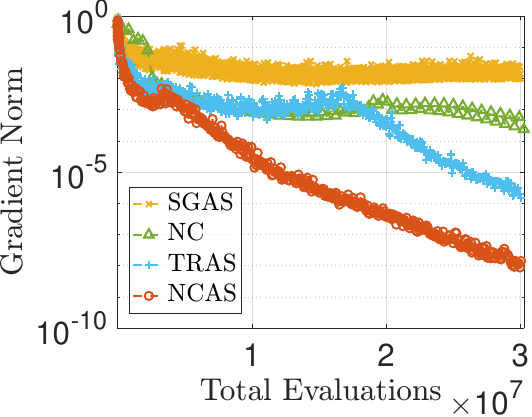}
    \end{subfigure}
    ~
    \begin{subfigure}{0.3\textwidth}
        \centering
        \includegraphics[width=0.99\textwidth]{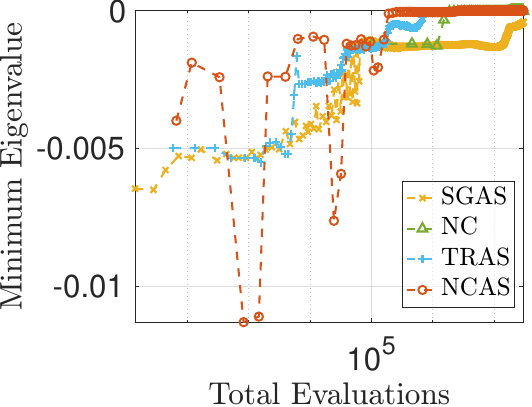}
    \end{subfigure}
    ~
    \begin{subfigure}{0.3\textwidth}
        \centering
        \includegraphics[width=\textwidth]{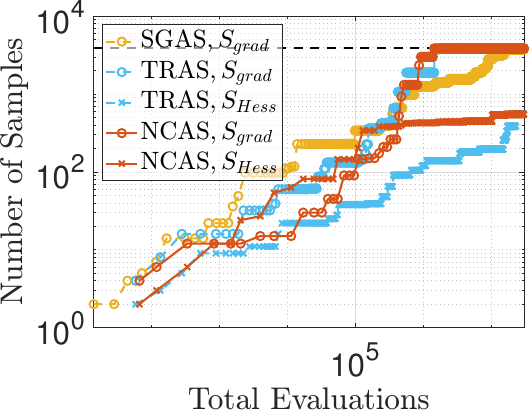}
    \end{subfigure}
    ~
    \begin{subfigure}{0.32\textwidth}
        \centering
        \includegraphics[width=\textwidth]{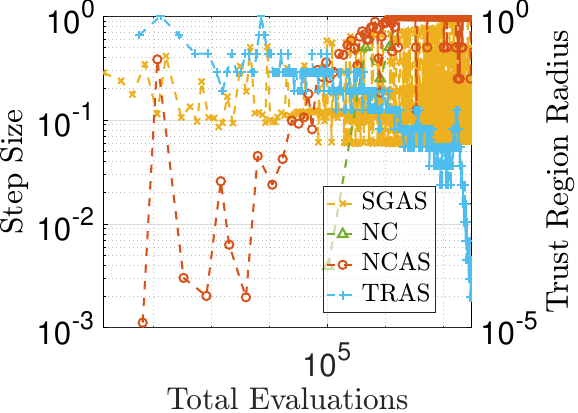}
    \end{subfigure}}\\
    \renewcommand\arraystretch{1.2}
    \centering\resizebox{\columnwidth}{!}{
    \begin{subfigure}{0.3\textwidth}
        \centering
        \includegraphics[width=\textwidth]{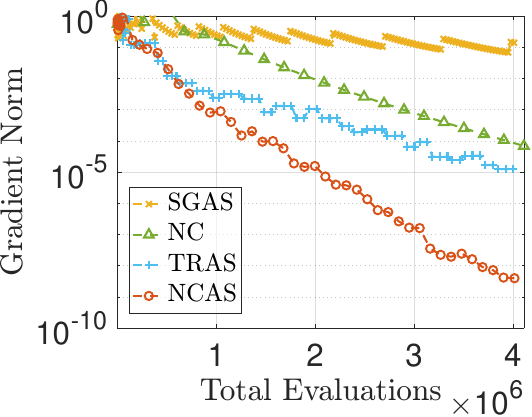}
    \end{subfigure}
    ~
    \begin{subfigure}{0.3\textwidth}
        \centering
        \includegraphics[width=\textwidth]{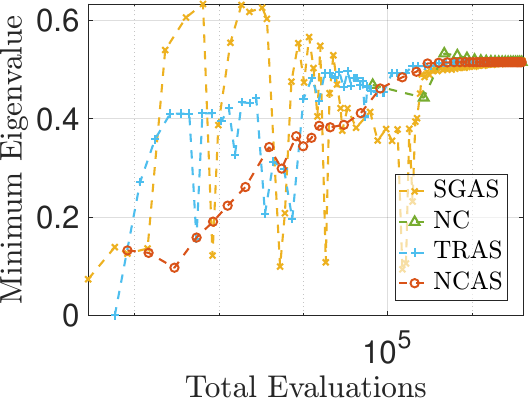}
    \end{subfigure}
    ~
    \begin{subfigure}{0.3\textwidth}
        \centering
        \includegraphics[width=\textwidth]{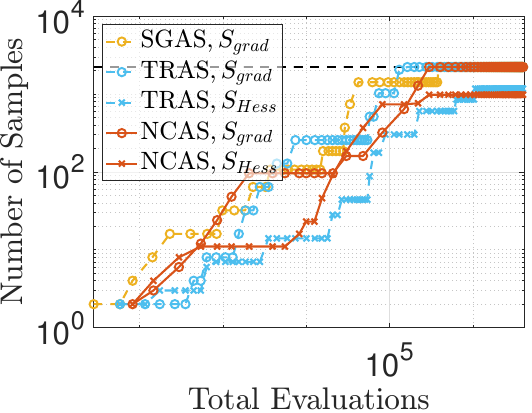}
    \end{subfigure}
    ~
    \begin{subfigure}{0.32\textwidth}
        \centering
        \includegraphics[width=\textwidth]{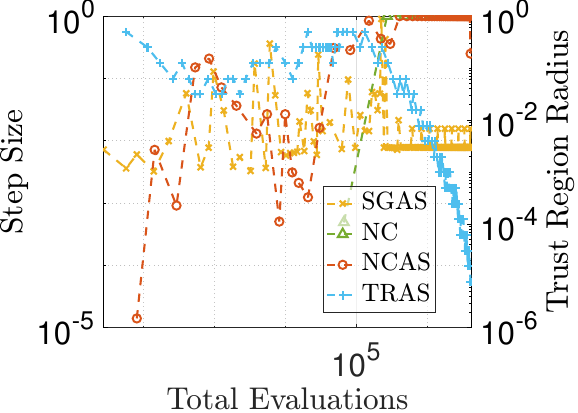}
    \end{subfigure}}
    \caption{Performance of \texttt{SGAS}, \texttt{NC}, \texttt{TRAS}, and \texttt{NCAS} on robust regression problem in terms of gradient norm, minimum eigenvalue, sample size, and step size/trust region radius with respect to total evaluations. First row \texttt{mushroom} dataset; Second row \texttt{splice} dataset.    }
    \label{fig.rr.mush}
\end{figure}

\newpage
\clearpage
\newpage

Next, we illustrate the performance of the four methods using the same datasets on the Tukey Biweight problem (Figures~\ref{fig.tb.aus} and \ref{fig.tb.mush}). Overall, the behavior of the methods is similar to that on the robust regression problems. Namely, the use of adaptive sampling allows for the frugal use of samples in the approximations employed, and the benefits of exploiting negative curvature are clear. That said, for these problems, the benefits are more pronounced. 
Our proposed algorithm \texttt{NCAS} strikes a good balance between the efficiency of samples used and the convergence speed and quality. It appears to efficiently converge to second-order stationary points for these three instances.

\begin{figure}[]
    \renewcommand\arraystretch{1.2}
    \centering\resizebox{\columnwidth}{!}{
    \begin{subfigure}{0.3\textwidth}
        \centering
        \includegraphics[width=\textwidth]{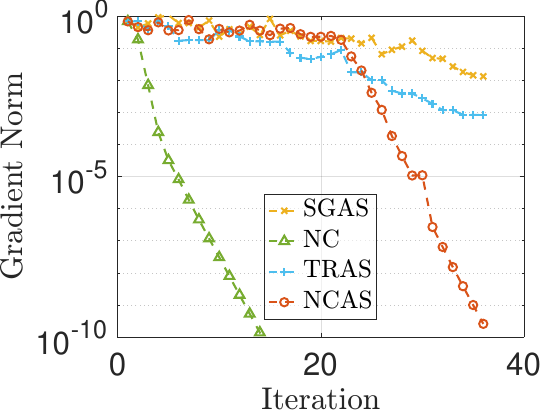}
    \end{subfigure}
    ~
    \begin{subfigure}{0.30\textwidth}
        \centering
        \includegraphics[width=\textwidth]{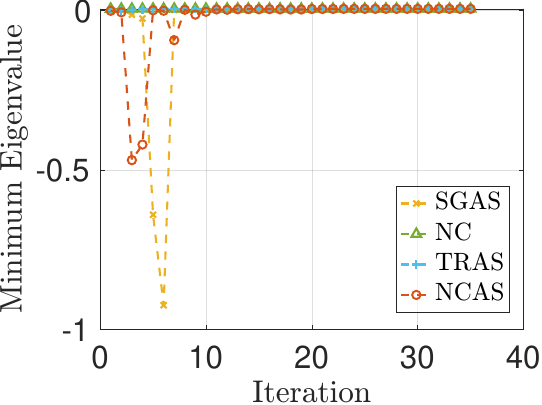}
    \end{subfigure}
    ~
    \begin{subfigure}{0.30\textwidth}
        \centering
        \includegraphics[width=0.99\textwidth]{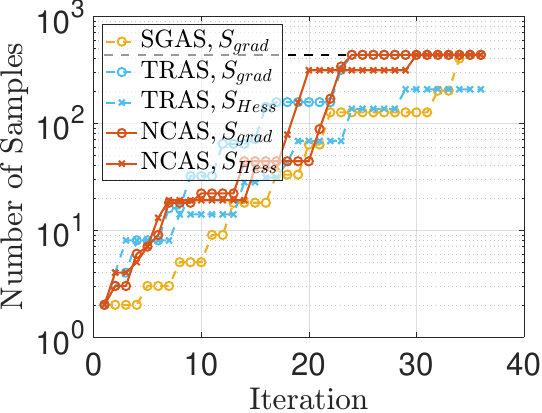}
    \end{subfigure}
    ~
    \begin{subfigure}{0.32\textwidth}
        \centering
        \includegraphics[width=\textwidth]{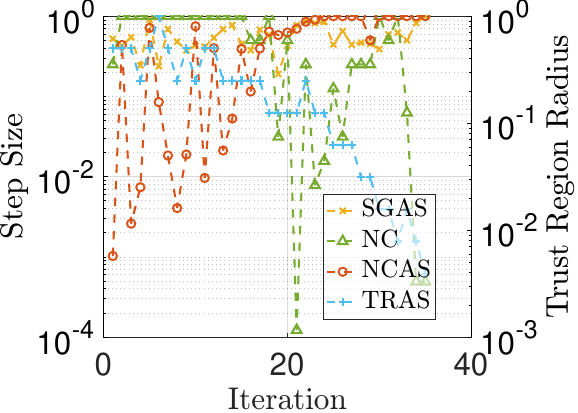}
    \end{subfigure}} \\
    \renewcommand\arraystretch{1.2}
    \centering\resizebox{\columnwidth}{!}{
    \begin{subfigure}{0.3\textwidth}
        \centering
        \includegraphics[width=\textwidth]{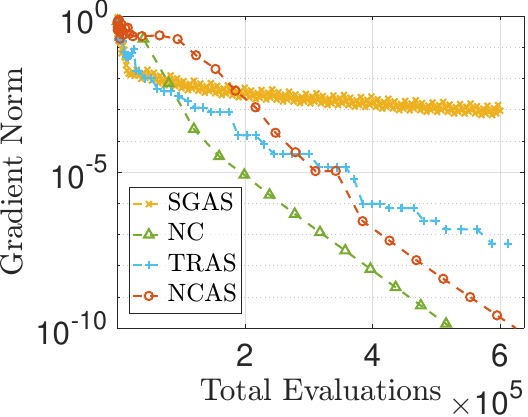}
    \end{subfigure}
    ~
    \begin{subfigure}{0.30\textwidth}
        \centering
        \includegraphics[width=\textwidth]{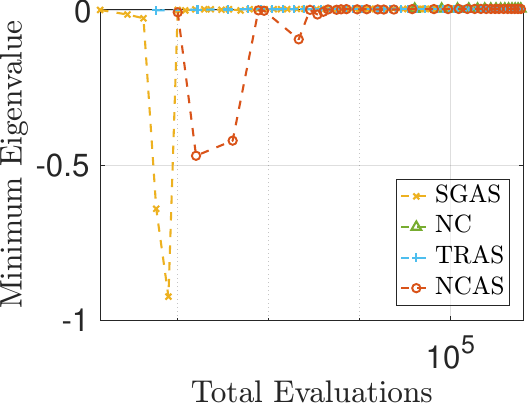}
    \end{subfigure}
    ~
    \begin{subfigure}{0.30\textwidth}
        \centering
        \includegraphics[width=0.99\textwidth]{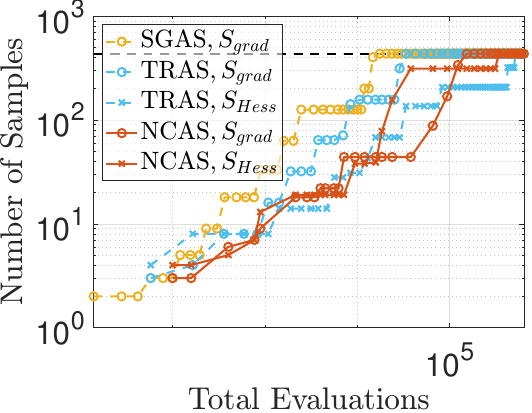}
    \end{subfigure}
    ~
    \begin{subfigure}{0.32\textwidth}
        \centering
        \includegraphics[width=\textwidth]{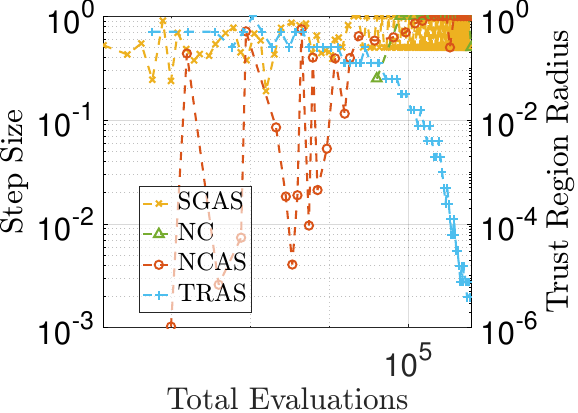}
    \end{subfigure}
    }
    \caption{Performance of \texttt{SGAS}, \texttt{NC}, \texttt{TRAS}, and \texttt{NCAS} on Tukey Biweight  problem (\texttt{australian} dataset) in terms of gradient norm, minimum eigenvalue, sample size, and step size/trust region radius. First row iterations; Second row total evaluations. 
    }\label{fig.tb.aus}
\end{figure}

\begin{figure}[h]    
    \renewcommand\arraystretch{1.2}
    \centering\resizebox{\columnwidth}{!}{
    \begin{subfigure}{0.295\textwidth}
        \centering
        \includegraphics[width=\textwidth]{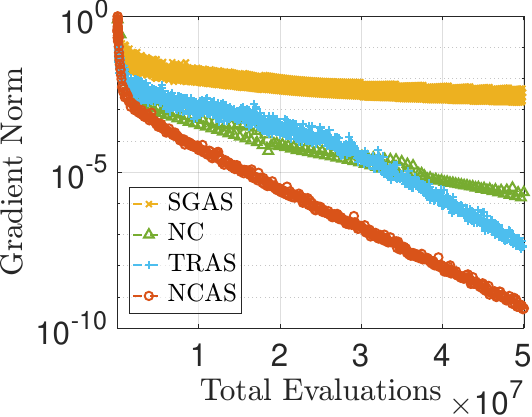}
    \end{subfigure}
    ~
    \begin{subfigure}{0.30\textwidth}
        \centering
        \includegraphics[width=\textwidth]{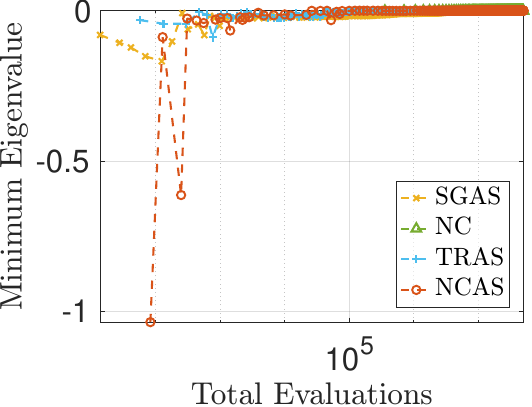}
    \end{subfigure}
    ~
    \begin{subfigure}{0.305\textwidth}
        \centering
        \includegraphics[width=\textwidth]{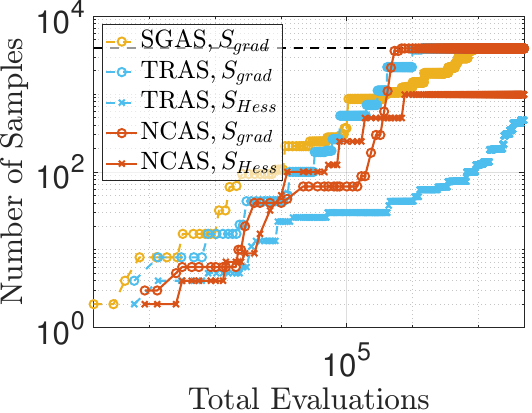}
    \end{subfigure}
    ~
    \begin{subfigure}{0.32\textwidth}
        \centering
        \includegraphics[width=\textwidth]{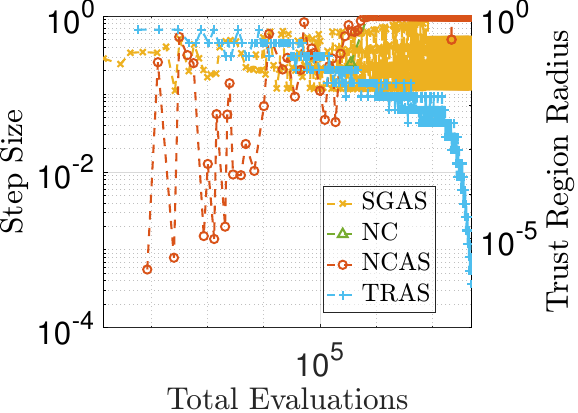}
    \end{subfigure}}\\
    
    \renewcommand\arraystretch{1.2}
    \centering\resizebox{\columnwidth}{!}{
    \begin{subfigure}{0.31\textwidth}
        \centering
        \includegraphics[width=0.95\textwidth]{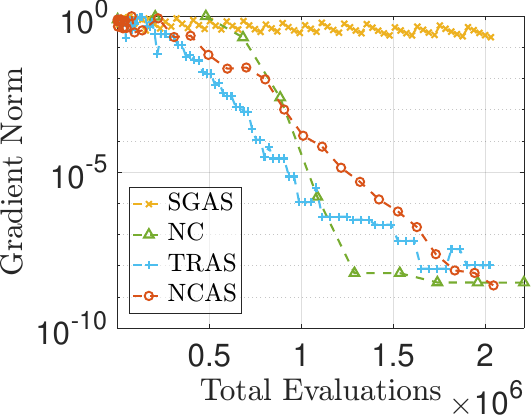}
    \end{subfigure}
    ~
    \begin{subfigure}{0.29\textwidth}
        \centering
        \includegraphics[width=\textwidth]{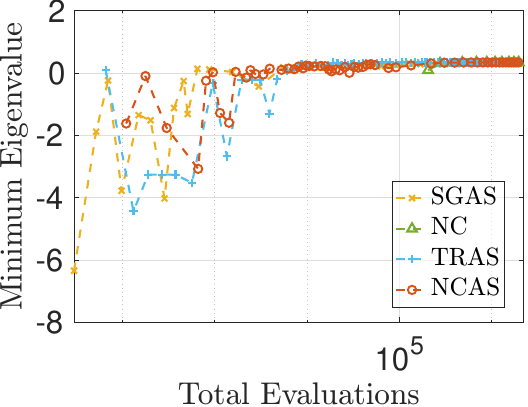}
    \end{subfigure}
    ~
    \begin{subfigure}{0.31\textwidth}
        \centering
        \includegraphics[width=0.95\textwidth]{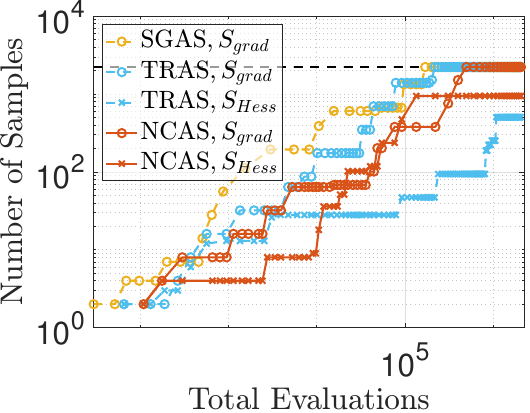}
    \end{subfigure}
    ~
    \begin{subfigure}{0.315\textwidth}
        \centering
        \includegraphics[width=\textwidth]{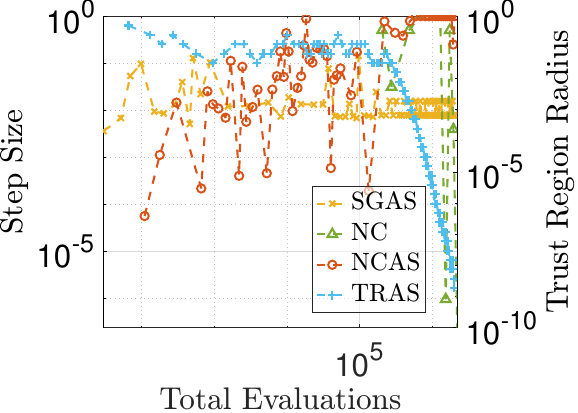}
    \end{subfigure}}
    \caption{Performance of \texttt{SGAS}, \texttt{NC}, \texttt{TRAS}, and \texttt{NCAS} on Tukey Biweight problem in terms of gradient norm, minimum eigenvalue, sample size, and step size/trust region radius with respect to total evaluations. First row \texttt{mushroom} dataset; Second row \texttt{splice} dataset.
    }\label{fig.tb.mush}
\end{figure}

\section{Final Remarks} \label{sec.fin_rem}

We have designed and analyzed a two-step method that incorporates 
both negative curvature and gradient information for minimizing general unconstrained smooth nonlinear optimization problems, applicable to both the deterministic inexact and stochastic settings. Our approach, under specific assumptions and conditions on the approximations employed, is endowed with convergence and complexity guarantees. Additionally, we have designed a practical variant of the method that utilizes the conjugate gradient method with negative curvature detection to compute a step, an adaptive sampling mechanism to decide on the approximate gradient and Hessian quality, and a dynamic step size selection procedure. Numerical experiments conducted on standard nonconvex regression problems highlight the efficiency, efficacy, and robustness of the proposed method.